\documentclass[12pt]{article}

\usepackage{amssymb,a4}
\usepackage{amsmath,amsfonts,amssymb,amsthm}

\newcommand{\Z}{{\mathbb Z}}

\newcommand{\C}{{\mathbb C}}

\newcommand{\N}{{\mathbb N}}

\newcommand{\A}{{\mathcal A}}
\newcommand{\B}{{\mathcal B}}
\newcommand{\E}{{\mathcal E}}
\newcommand{\D}{{\mathcal D}}

\def\<{\langle}
\def\>{\rangle}

\def\g{\mathfrak g}

\def\End{\mathrm{End}}
\def\Hom{\mathrm{Hom}}

\newtheorem{thm}{Theorem}[section]
\newtheorem{prop}[thm]{Proposition}
\newtheorem{lem}[thm]{Lemma}
\newtheorem{cor}[thm]{Corollary}
\newtheorem{rmk}[thm]{Remark}
\newtheorem{definition}[thm]{Definition}

\begin{document}

\makeatletter \@addtoreset{equation}{section}
\def\theequation{\thesection.\arabic{equation}}
\makeatother \makeatletter

\begin{center}
{\Large \bf   Ding-Iohara algebras and quantum vertex algebras}
\end{center}

\begin{center}
{Haisheng Li$^{a,b}$\footnote{Partially supported by
 China NSF grant (Nos.11471268, 11571391)},
Shaobin Tan$^{b}$\footnote{Partially supported by China NSF grant (Nos.11471268, 11531004)}
and Qing Wang$^{b}$\footnote{Partially supported by
 China NSF grants (Nos.11531004, 11622107), Natural Science Foundation of Fujian Province
(No. 2016J06002)}\\
$\mbox{}^{a}$Department of Mathematical Sciences\\
Rutgers University, Camden, NJ 08102, USA\\
$\mbox{}^{b}$School of Mathematical Sciences, Xiamen University,
Xiamen 361005, China}
\end{center}

\begin{abstract}
In this paper, we associate quantum vertex algebras to  a certain family of associative algebras $\widetilde{\A}(g)$
which are essentially Ding-Iohara algebras. To do this, we introduce another closely related family
of associative algebras $\A(h)$. The associated quantum vertex algebras are based on the vacuum modules
for $\A(h)$, whereas $\phi$-coordinated modules for these quantum vertex algebras are associated to
$\widetilde{A}(g)$-modules. Furthermore, we classify their irreducible $\phi$-coordinated modules.
\end{abstract}

\section{Introduction}
This paper is  from a vertex algebra point of view to study a family of associative algebras $\widetilde{\A}(g)$
with $g(z)$ a rational function such that $g(z)g(1/z)=1$.
By definition, $\widetilde{\A}(g)$ is the associative unital algebra over $\mathbb{C}$, generated by
$$E_{n},\   F_{n},\    \Psi_{n}\   \   (n\in \mathbb{Z}),$$
 subject to relations
\begin{eqnarray*}
&&E(z)E(w)=\iota_{w,z}(g(w/z))E(w)E(z), \    \   \
F(z)F(w)=\iota_{w,z}(g(z/w))F(w)F(z),\label{eq:5.1}\\
&&\Psi(z)E(w)=\iota_{w,z}(g(w/z))E(w)\Psi(z), \    \   \
\Psi(z)F(w)=\iota_{w,z}(g(z/w))F(w)\Psi(z),\label{eq:5.2}\\
&&[E(z),F(w)]=\delta\left(\frac{w}{z}\right)\Psi(w),\   \   \   [\Psi(z),\Psi(w)]=0,
\end{eqnarray*}
where
$$E(z)=\sum_{n\in \mathbb{Z}}E_{n}z^{-n},\   \   \   F(z)=\sum_{n\in \mathbb{Z}}F_{n}z^{-n},\   \  \
\Psi(z)=\sum_{n\in \mathbb{Z}}\Psi_{n}z^{-n}.$$
Note that in the special case with $$g(z)=\prod_{j=1,2,3}\frac{1-q_j^{-1}z}{ 1-q_jz},$$
where $q_1,q_2,q_3$ are nonzero complex numbers such that $q_1q_2q_3=1$,
 $\widetilde{\A}(g)$ is (essentially) the algebra $\E$ which had appeared in \cite{bfmzz} (with $h(z)=g(z)$).
On the other hand, these algebras are essentially Ding-Iohara algebras of level zero (see \cite{ding-iohara}).

Ding-Iohara algebras are a family of Hopf algebras, generalizing quantum affine algebras $U_{q}(\widehat{\g})$.
In the case $\g=sl_2$, Ding-Iohara algebras, which are parametrized by a rational function $g(z)$ such that
$g(z)g(1/z)=1$, are generated by the modes of fields $e(z), f(z), \psi^{\pm }(z)$, and by invertible central elements
$\gamma^{\pm 1/2}$.
In the past, Ding-Iohara algebras had been studied by many people in various directions (see for example \cite{FT}, \cite{afhksy}).

In this paper, we study Ding-Iohara algebras $\widetilde{\A}(g)$ from a vertex-algebra point of view
and our main goal is to establish a natural connection of these algebras
with vertex algebras or more generally quantum vertex algebras in some sense.
As the main results of this paper, we  associate quantum vertex algebras  and
 their $\phi$-coordinated modules in the sense of  \cite{Li-nonlocal} and \cite{Li-phi}
 to the Ding-Iohara algebras $\widetilde{\A}(g)$.

In literature, there have been several theories of quantum vertex (operator) algebras, where the better known
representatives are the (Edward) Frenkel-Reshetikhin theory of deformed chiral algebras (see \cite{FR}),
 the Etingof-Kazhdan theory of quantum vertex operator algebras (see \cite{EK}),  and the Borcherds theory of quantum vertex algebras (see \cite{Bqva}). Each of these theories, which are different in certain ways, has its own  interest.
While these pioneer works provide important foundations in the study on
quantum vertex algebras,  the general theory is yet to be fully developed.

For many years, we have been extensively studying quantum vertex algebras
 (see \cite{Li-nonlocal, Li-const, Li-phi, Li6}, \cite{KL}, \cite{LTW}), essentially along the line of Etingof-Kazhdan's theory.
In nature, Etingof-Kazhdan's quantum theory is in the sense of formal deformation, where
quantum vertex operator algebras in this sense are formal deformations of vertex algebras.
Note that among the important properties of vertex algebras are
(weak) associativity and commutativity (namely locality).  For quantum vertex operator algebras,
 (weak) associativity is postulated while locality is replaced with what was called $S$-locality (see \cite{EK}).
 Motivated by Etingof-Kazhdan's theory,  we developed a theory of (weak) quantum vertex algebras, where
weak quantum vertex algebras, instead of being formal deformations of vertex algebras,  are generalizations of
vertex algebras and vertex super-algebras.
Just as with vertex algebras,  for a general weak quantum vertex algebra one has the notion of module
(and that of twisted module (see \cite{LTW})).
A conceptual result (see \cite{Li-nonlocal}) is that every $S$-local set of vertex operators (namely fields)
on a general vector space $W$ generates a weak quantum vertex algebra in a certain canonical way,
with $W$ as a natural module.

Note that  Anguelova and Bergvelt developed a theory of what were called $H_{D}$-quantum vertex algebras
(see \cite{AB}).

To associate quantum vertex algebras to certain algebras such as quantum affine algebras,
a theory of what were called $\phi$-coordinated (quasi) modules for weak quantum vertex algebras
was developed  in \cite{Li-phi}. In this theory, $\phi$  is what was called therein an associate of
the $1$-dimensional additive formal group (law), which is $F_{\bf a}(x,y)=x+y\in \C[[x,y]]$.
By definition, an associate of $F_{\bf a}$ is
a formal series $\phi(x,z)\in \C((x))[[z]]$ such that
$$\phi(x,0)=x,\   \   \phi(\phi(x,y),z))=\phi(x,y+z).$$
With $F_{\bf a}$ clearly being an associate for itself, it was proved that all associates of
$F_{\bf a}$ can be obtained by
$$\phi(x,z)=e^{zp(x)\frac{d}{dx}}\cdot x,$$
where $p(x)\in \C((x))$. Note that $\phi(x,z)=F_{\bf a}(x,z)$ for $p(x)=1$ and
$\phi(x,z)=xe^{z}$ for $p(x)=x$, another particular associate.
The essence is that the usual associativity, which is governed by the formal group law $F_{\bf a}$,
 is generalized to $\phi$-associativity, which is governed by a general associate $\phi$ of $F_{\bf a}$.
Let $\phi(x,z)$ be an associate of $F_{\bf a}$. For a weak quantum vertex algebra  $V$,
a $\phi$-coordinated $V$-module by definition is a vector space $W$ equipped with a linear map
$$Y_{W}(\cdot,x): V\rightarrow (\End W)[[x,x^{-1}]]; \ v\mapsto Y_{W}(v,x),$$
satisfying the conditions that $Y_{W}(v,x)w\in W((x))$ for $v\in V,\ w\in W$, $Y_{W}({\bf 1},x)=1$, and
that {\em weak $\phi$-associativity} holds: For any $u,v\in V$, there exists a nonnegative integer $k$ such that
$$(x_1-x_2)^{k}Y_{W}(u,x_1)Y_{W}(v,x_2)\in \Hom (W,W((x_1,x_2))),$$
$$(\phi(x_2,z)-x_2)^{k}Y_{W}(Y(u,z)v,x_2)=\left((x_1-x_2)^{k}Y_{W}(u,x_1)Y_{W}(v,x_2)\right)|_{x_1=\phi(x_2,z)}.$$
In this paper, we shall associate $\phi$-coordinated modules
with $\phi(x,z)=xe^z$ for certain quantum vertex algebras to the Ding-Iohara algebras $\widetilde{\A}(g)$.

First, we construct the desired quantum vertex algebras.
For this, we introduce certain counterparts of  $\widetilde{\A}(g)$, another family of associative algebras.
Let $h(z)\in {\mathbb{C}}[[z]]$ such that $h(z)h(-z)=1$. For example,
 \begin{eqnarray*}
 h(x)=\iota_{x,0}(g(e^{x}))\in \C((x)),
 \end{eqnarray*}
 where $g(z)$ is a rational function as before and $\iota_{x,0}(g(e^{x}))$ denotes the formal Laurent series expansion
 of $g(e^{x})$ at $x=0$.
We define $\A(h)$ to be the associative unital algebra over $\mathbb{C}$ with generators
$$e_{n},\   f_{n},\    \psi_{n}\   \   (n\in \mathbb{Z}),$$
 subject to relations
\begin{eqnarray*}
&&e(z)e(w)=h(w-z)e(w)e(z), \    \   \    f(z)f(w)=h(z-w)f(w)f(z),\label{eqee}\\
&&\psi(z)e(w)=h(w-z)e(w)\psi(z), \    \  \   \psi(z)f(w)=h(z-w)f(w)\psi(z),\label{eqpsie}\\
&&[e(z),f(w)]=z^{-1}\delta\left(\frac{w}{z}\right)\psi(z),  \   \   \
[\psi(z),\psi(w)]=0,\label{eqpsipsi}
\end{eqnarray*}
where $a(x)=\sum_{n\in \mathbb{Z}}a_{n}x^{-n-1}$ for $a=e,f,\psi$. Let $V_{\A(h)}$ be the vacuum $\A(h)$-module
in the sense that $V_{\A(h)}$ is the $\A(h)$-module generated by a vector  ${\bf 1}$, called the vacuum vector, such that
$e_{n}{\bf 1}=f_{n}{\bf 1}=\psi_{n}{\bf 1}=0$ for all $n\ge 0$.
Then we show that there exists a weak quantum vertex algebra structure in the sense of \cite{Li-nonlocal}
on $V_{\A(h)}$. By making use of an affine vertex (super)algebra we construct $V_{\A(h)}$  and determine a basis of P-B-W type.
We show that the associated  weak quantum vertex algebras are non-degenerate in the sense of
Etingof-Kazhdan, proving that  they are quantum vertex algebras.  On the other hand, we show that a
suitably defined restricted $\widetilde{\A}(g)$-module amounts to a $\phi$-coordinated $V_{\A(h)}$-module.

This paper is organized as follows: Section 2 is preliminaries; In this section we recall basic notions and results about  (weak) quantum vertex algebras, their modules, and their $\phi$-coordinated modules.
In Section 3,  we introduce the associative algebra $\A(h)$ for each series $h(x)\in \C[[x]]$ with $h(x)h(-x)=1$,
and we construct a weak quantum vertex algebra $V_{\A(h)}$.
In Section 4, we determine the structure of the weak quantum vertex algebra $V_{\A(h)}$. In particular,
we prove that they are non-degenerate in the sense of Etingof-Kazhdan.
In Section 5, for each rational function $g(z)$ with $g(z)g(1/z)=1$, we  introduce an associative algebra
$\widetilde{\A}(g)$
and we identify suitably defined restricted $\widetilde{\A}(g)$-modules with $\phi$-coordinated modules for the quantum vertex algebra $V_{\A(h)}$ with $h(x)=g(e^{x})$.

\section{Preliminaries}
In this section,  we recall from \cite{Li-nonlocal} and \cite{Li-phi} some basic notations and results on quantum vertex algebras
and their modules, including the conceptual construction of (weak) quantum vertex algebras and
 modules.

Throughout this paper, $\N$ denotes the set of nonnegative integers,
 $\mathbb{C}^{\times}$ denotes the multiplicative group of nonzero complex numbers
 (while $\C$ denotes the complex number field), and
the symbols $x,y,x_{1},x_{2},\dots $ denote mutually commuting independent formal variables.
All vector spaces in this paper are
considered to be over  $\mathbb{C}$.

For a vector space $U$, $U((x))$ is the vector space of lower
truncated integer power series in $x$ with coefficients
 in $U$, $U[[x]]$ is the vector space of nonnegative integer
 power series in $x$ with coefficients in $U$, and
$U[[x,x^{-1}]]$ is the vector space of doubly infinite integer power series in $x$ with coefficients in $U$.

We now begin by recalling the definitions of nonlocal vertex algebra and module
(see \cite{Li-nonlocal}, \cite{Li-g1}; cf. \cite{Bva}, \cite{BK}).

\begin{definition}\label{defnlva}
{\em A {\em nonlocal vertex algebra} is a vector space $V$ equipped with a linear map
      \begin{eqnarray*}
             Y(\cdot,x) :&&V\longrightarrow \mathrm{Hom}(V,V((x)))\subset (\mathrm{EndV})[[x,x^{-1}]]\\
              &&v\longmapsto Y(v,x)=\sum_{n\in\mathbb{Z}}v_{n}x^{-n-1}\ \ (\mbox{where }v_{n}\in \End V)
      \end{eqnarray*}
              and equipped with a distinguished vector $\textbf{1}\in V$, called the {\em vacuum vector},
              satisfying the conditions that
              $$Y(\textbf{1},x)v=v,$$
              $$Y(v,x)\textbf{1}\in V[[x]] \;\;\mbox{and}\;\; \lim_{x\rightarrow  0}Y(v,x)\textbf{1} = v\   \   \mbox{ for }v\in V,$$
              and  that for $u,v,w\in V$, there exists a nonnegative integer $l$ such that
       \begin{eqnarray*}
             (x_0+x_2)^{l}Y(u,x_{0}+x_2)Y(v,x_{2})w=(x_0+x_2)^{l}Y(Y(u,x_0)v,x_2)w.
         \end{eqnarray*}}
\end{definition}

\begin{definition}\label{defmodule}
{\em Let $V$ be a nonlocal vertex algebra. A {\em $V$-module} is a vector space
                  $W$ equipped with a linear map
   \begin{eqnarray*}
     Y_{W}(\cdot,x) :&&V\longrightarrow \mathrm{Hom}(W,W((x)))\subset (\mathrm{EndW})[[x,x^{-1}]]\\
              && v\longmapsto Y_{W}(v,x)=\sum_{n\in\mathbb{Z}}v_{n}x^{-n-1}\ \ (\mbox{where }v_{n}\in \End W),
   \end{eqnarray*}
              satisfying the conditions that
                   $$ Y_{W}(1,x) = 1_{W}\ \ (\mbox{the identity operator on }W)$$
                    and  that for $u,v\in V$, $w\in W$, there exists a nonnegative integer $l$
                    such that
                    \begin{eqnarray*}
              &&(x_0+x_2)^{l}Y_{W}(u,x_{0}+x_2)Y_{W}(v,x_{2})w=(x_0+x_2)^{l}Y_{W}(Y(u,x_0)v,x_2)w.
                         \end{eqnarray*}}
\end{definition}

The last condition in Definitions \ref{defnlva} and \ref{defmodule} is often referred to as {\em weak associativity}.

Recall from \cite{Li-nonlocal}  the following notion of weak quantum vertex algebra:

\begin{definition}
{\em A {\em weak quantum vertex algebra} is a vector space $V$ equipped with a linear map
             $$Y(\cdot,x) :V\rightarrow \mathrm{Hom}(V,V((x)))\subset (\End V)[[x,x^{-1}]]$$
             and a vector $\textbf{1}\in V$,
              satisfying the conditions that for $v\in V$,
              $$Y(\textbf{1},x)v=v,$$
              $$Y(v,x)\textbf{1}\in V[[x]] \;\;\mbox{and}\;\; \lim_{x\rightarrow  0}Y(v,x)\textbf{1} = v,$$
 and that for $u,v\in V$, there exists $\sum_{i=1}^{r}v^{(i)}\otimes u^{(i)}\otimes f_{i}(x)\in V\otimes V\otimes {\mathbb{C}}((x))$
 such that
 \begin{eqnarray}\label{eSjacobi}
  &&x_{0}^{-1}\delta\left(\frac{x_{1}-x_{2}}{x_{0}}\right)Y(u,x_{1})Y(v,x_{2}) -
     x_{0}^{-1}\delta\left(\frac{x_{2}-x_{1}}{-x_{0}}\right)\sum_{i=1}^{r}f_{i}(-x_{0})Y(v^{(i)},x_{2})Y(u^{(i)},x_{1})\nonumber\\
 &&\hspace{3cm}= x_{2}^{-1}\delta\left(\frac{x_{1}-x_{0}}{x_{2}}\right)Y(Y(u,x_{0})v,x_{2}) 
 \end{eqnarray}
 (the {\em $\mathcal{S}$-Jacobi identity}).}
\end{definition}

Note that  the $\mathcal{S}$-Jacobi identity (\ref{eSjacobi}) implies the weak associativity,
so that a weak quantum vertex algebra is automatically a nonlocal vertex algebra.
On the other hand, it is clear that the notion of weak quantum vertex algebra generalizes that of
vertex algebra and vertex super-algebra.

For a weak quantum vertex algebra $V$, a {\em $V$-module} is defined to be a module for $V$
viewed as a nonlocal vertex algebra.
The following was proved in \cite{Li-g1}:

\begin{prop}
Let $V$ be a weak quantum vertex algebra and let $(W,Y_{W})$ be any $V$-module.
Then, for $u,v\in V$, whenever (\ref{eSjacobi}) holds, we have
 \begin{eqnarray*}
  &&x_{0}^{-1}\delta\left(\frac{x_{1}-x_{2}}{x_{0}}\right)Y_{W}(u,x_{1})Y_{W}(v,x_{2}) -
     x_{0}^{-1}\delta\left(\frac{x_{2}-x_{1}}{-x_{0}}\right)\sum_{i=1}^{r}f_{i}(-x_{0})Y_{W}(v^{(i)},x_{2})Y(u^{(i)},x_{1})\\
  &&\hspace{3cm}= x_{2}^{-1}\delta\left(\frac{x_{1}-x_{0}}{x_{2}}\right)Y_{W}(Y(u,x_{0})v,x_{2}).
\end{eqnarray*}
\end{prop}

Recall that a {\em rational quantum Yang-Baxter operator} on a vector space $U$ is a linear map
 $$S(x):\   U\otimes U\rightarrow U\otimes U\otimes \mathbb{C}((x)),$$
 satisfying $$S^{12}(x)S^{13}(x+z)S^{23}(z)=S^{23}(z)S^{13}(x+z)S^{12}(x)$$
 (the {\em quantum Yang-Baxter equation}), where for $1\leq i<j\leq 3$,
 $$S^{ij}(x): U\otimes U\otimes U\rightarrow U\otimes U\otimes U\otimes \mathbb{C}((x))$$
 denotes the canonical extension of $S(x)$. It is said to be {\em unitary} if $S(x)S^{21}(-x)=1,$
 where $S^{21}(x)=\sigma S(x)\sigma$ with $\sigma$ denoting the flip operator on $U\otimes U$.

For a nonlocal vertex algebra $V$, following \cite{EK}, denote by $Y(x)$ the linear map
$$Y(x):\ V\otimes V\rightarrow V((x)),$$
associated to the linear map $Y(\cdot,x)$.
The following notion of quantum vertex algebra was introduced in \cite{Li-nonlocal} (cf. \cite{EK}):

\begin{definition}
{\em A {\em quantum vertex algebra} is a weak quantum vertex algebra $V$ equipped with a unitary rational quantum Yang-Baxter operator ${\mathcal{S}}(x)$ on $V$ such that for $u,v\in V$, (\ref{eSjacobi}) holds with
$\sum_{i=1}^{r}v^{(i)}\otimes u^{(i)}\otimes f_{i}(x)={\mathcal{S}}(x)(v\otimes u)$ and such that
\begin{eqnarray}
&&{\mathcal{S}}(x)({\bf 1}\otimes v)={\bf 1}\otimes v\ \ \ \mbox{ for }v\in
V,\label{esvacuum}\\
&&[\D\otimes 1, {\mathcal{S}}(x)]=-\frac{d}{dx}{\mathcal{S}}(x),\label{d1s}\\
&&{\mathcal{S}}(x_{1})(Y(x_{2})\otimes 1)=(Y(x_{2})\otimes
1){\mathcal{S}}^{23}(x_{1}){\mathcal{S}}^{13}(x_{1}+x_{2}).\label{sy1}
\end{eqnarray}}
\end{definition}

The following notion is due to Etingof and Kazhdan (see \cite{EK}):

\begin{definition}
{\em A nonlocal vertex algebra $V$ is said to be {\em non-degenerate} if for every positive integer $n$, the linear map
$$Z_n:\  \mathbb{C}((x_{1}))\cdots((x_{n}))\otimes V^{\otimes n}\rightarrow V((x_{1}))\cdots((x_{n}))$$ defined by
$$Z_n(f\otimes v^{(1)}\otimes\cdots\otimes v^{(n)})=fY(v^{(1)},x_{1})\cdots Y(v^{(n)},x_{n})\textbf{1}$$
is injective.}
\end{definition}

The following was proved in \cite{Li-nonlocal}:

\begin{prop}\label{pnondeg-qva}
Every non-degenerate weak quantum vertex algebra is a quantum vertex algebra
with a uniquely determined rational quantum Yang-Baxter operator.
\end{prop}

\begin{rmk}
{\em  In view of Proposition \ref{pnondeg-qva}, the term ``non-degenerate quantum vertex algebra''
without specifying a quantum Yang-Baxter operator is unambiguous.
It was proved  in \cite{Li-nonlocal} that if $V$ is of countable dimension (over $\C$)
and if $V$ as a (left) $V$-module is irreducible, then $V$ is non-degenerate.
Then the term ``irreducible quantum vertex algebra''
without specifying a quantum Yang-Baxter operator is unambiguous.}
\end{rmk}

Let $W$ be a general vector space. Set
\begin{eqnarray*}
\mathcal{E}(W)=\mbox{Hom}(W , W((x)))\subset(\mbox{EndW})[[x,x^{-1}]].
\end{eqnarray*}
The identity operator on $W$, denoted by $\textbf{1}_{W}$, is a special
element of $\mathcal{E}(W).$

\begin{definition} {\em A subset $U$ of $\mathcal{E}(W)$ is said to be {\em $S$-local}
if for any $a(x), b(x)\in U$, there exist
$$a_{i}(x), b_{i}(x)\in U, \ f_{i}(x)\in \mathbb{C}((x)), \ i=1,\ldots,r,$$
and a nonnegative integer $k$ such that
\begin{eqnarray}\label{eS-locality-relation}
(x_1-x_2)^{k}a(x_{1})b(x_{2})=(x_1-x_2)^{k}\sum_{i=1}^{r}f_{i}(x_{2}-x_{1})b_{i}(x_2)a_{i}(x_1).
\end{eqnarray}}
\end{definition}

Let $W$ be a vector space as before and let $U$ be any $S$-local subset of $\mathcal{E}(W)$.
 Assume $a(x),b(x)\in U$.
Notice that the relation (\ref{eS-locality-relation}) implies
 \begin{eqnarray}\label{ecompat}
 (x_{1}-x_2)^{k}a(x_{1})b(x_{2})\in \mbox{Hom}(W, W((x_{1},x_{2}))).
 \end{eqnarray}
 Define $a(x)_{n}b(x)\in\mbox{(End W)}[[x,x^{-1}]]$ for $n\in \Z$
                  in terms of generating function
                 \begin{eqnarray}Y_{\mathcal{E}}(a(x),z)b(x)=
                  \sum_{n\in\mathbb{Z}}(a(x)_{n}b(x))z^{-n-1}   \label{eq:2.4}\end{eqnarray}
                   by
 \begin{eqnarray}
  Y_{\mathcal{E}}(a(x),z)b(x)=\mbox{Res}_{x_{1}}x^{-1}\delta\left(\frac{x_{1}-z}{x}\right)x_{0}^{-k}((x_1-x)^{k}a(x_1)b(x)),
 \end{eqnarray}
               where $k$ is any nonnegative integer such that (\ref{ecompat}) holds.
               Assuming the $S$- locality relation (\ref{eS-locality-relation}), we have
                 \begin{eqnarray}
                 Y_{\mathcal{E}}(a(x),z)b(x)
                 &=&\mbox{Res}_{x_{1}}z^{-1}\delta\left(\frac{x_{1}-x}{z}\right)a(x_1)b(x)\nonumber\\
                 &&-\mbox{Res}_{x_{1}}z^{-1}\delta\left(\frac{x-x_{1}}{-z}\right)\sum_{i=1}^{r}\iota_{x,x_{1}}(f_{i}(x-x_1))b^{(i)}(x)a^{(i)}(x_1), \label{eq:2.5}
                 \end{eqnarray}
                 or equivalently, for $n\in\Z$,
                 $$a(x)_{n}b(x)=\mbox{Res}_{x_{1}}\left((x_1-x)^{n}a(x_1)b(x)
                 -\sum^{r}_{i=1}(-x+x_{1})^{n}f_{i}(x-x_1)b^{(i)}(x)a^{(i)}(x_1)\right).$$

Let $U$ be an $S$-local subspace of $\mathcal{E}(W)$. We say
 $U$ is {\em $Y_{\mathcal{E}}$-closed} if
           \begin{eqnarray*}
                 a(x)_{n}b(x)\in U
                  \    \    \mbox{ for all }a(x),b(x)\in U,\ n\in\mathbb{Z}.
            \end{eqnarray*}

The following result was obtained in \cite{Li-const}:

\begin{thm}\label{thm2.1}
Let $W$ be a vector space and let $U$ be any $S$-local subset of $\mathcal{E}(W)$.
Then there exists a $Y_{\mathcal{E}}$-closed
$S$-local subspace of $\mathcal{E}(W)$, which contains $U$ and $\textbf{1}_{W}$.
Denote by $\langle U\rangle$ the smallest such subspace.
Then $(\langle U\rangle,Y_{\mathcal{E}},\textbf{1}_{W})$
carries the structure of a weak quantum vertex algebra and  $W$ is a faithful
  $\langle U\rangle$-module with $Y_{W}(a(x),z) = a(z)$ for $a(x)\in\langle U\rangle.$
\end{thm}

Next, we recall from \cite{Li-phi} and \cite{Li6} some basic results in the theory of $\phi$-coordinated modules for weak quantum vertex algebras. In this theory, $\phi$ stands for the formal series $\phi(x,z)=xe^{z}\in \mathbb{C}[[x,z]]$, which is what was called therein an associate of the $1$-dimensional additive formal group (law) $F(x,y)=x+y$.

\begin{definition}
{\em Let $V$ be a weak quantum vertex algebra. A {\em $\phi$-coordinated $V$-module} is a vector space
                  $W$ equipped with a linear map
                  $$Y_{W}(\cdot,x) :\  V\longrightarrow \mathrm{Hom}(W,W((x)))\subset (\mathrm{EndW})[[x,x^{-1}]],$$
              satisfying the conditions that $Y_{W}({\bf 1},x) = 1_{W}$
                    and that for any $u,v\in V$, there exists a nonnegative integer $k$ such that
                    \begin{eqnarray*}
              &&(x_1-x_2)^{k}Y_{W}(u,x_{1})Y_{W}(v,x_{2})\in \mathrm{Hom}(W,W((x_{1},x_{2})))
                         \end{eqnarray*}
                         and
                    \begin{eqnarray*}
              &&(x_{2}e^{z}-x_2)^{k}Y_{W}(Y(u,z)v,x_{2})=((x_1-x_2)^{k}Y_{W}(u,x_{1})Y_{W}(v,x_{2})|_{x_1=x_{2}e^{z}}.
                         \end{eqnarray*}}
\end{definition}

Let $\mathbb{C}(x)$ and $\C(x_1,x_2)$ denote the fields of rational functions. Define
\begin{eqnarray}
\iota_{x_{1},x_{2}}: \  \C(x_{1},x_{2})\rightarrow \C((x_{1}))((x_{2}))
\end{eqnarray}
to be the canonical extension of the ring embedding of $\C[x_{1},x_{2}]$ into the field $\C((x_{1}))((x_{2}))$.

The following result was obtained in \cite{Li-phi}:

\begin{prop}
Let $V$ be a weak quantum vertex algebra and let $(W,Y_{W})$ be a $\phi$-coordinated $V$-module.
Let $u,v\in V$ and suppose that $u^{(i)},\ v^{(i)}\in V,\ q_i(x)\in \C(x)\ (1\le i\le r)$ such that
\begin{eqnarray}
(x_1-x_2)^{k}Y(u,x_1)Y(v,x_2)=(x_1-x_2)^{k}\sum_{i=1}^{r}\iota_{x_2,x_1}(q_i(e^{x_2-x_1}))Y(v^{(i)},x_2)Y(u^{(i)},x_1)
\end{eqnarray}
on $V$ for some nonnegative integer $k$. Then
\begin{eqnarray}
&&Y_{W}(u,x_1)Y_W(v,x_2)-\sum_{i=1}^{r}\iota_{x_2,x_1}(q_{i}(x_2/x_1))Y_{W}(v^{(i)},x_2)Y_{W}(u^{(i)},x_1)
\nonumber\\
&=&\sum_{n\ge 0}\frac{1}{n!}Y_{W}(u_nv,x_2)\left(x_2\frac{\partial}{\partial x_2}\right)^{n}\delta\left(\frac{x_2}{x_1}\right).
\end{eqnarray}
\end{prop}

\begin{definition}
{\em Let $W$ be a vector space.  A subset $U$ of $\mathcal{E}(W)$ is said to be $S_{trig}$-local
if for any $a(x), b(x)\in U$, there exist
$$u_{i}(x), v_{i}(x)\in U, \ q_{i}(x)\in \mathbb{C}(x), \ i=1,\ldots,r,$$
 such that
\begin{eqnarray}
(x_1-x_2)^{k}a(x_{1})b(x_{2})=(x_1-x_2)^{k}\sum_{i=1}^{r}\iota_{x_2,x_1}(q_{i}(x_{1}/x_{2}))u_{i}(x_2)v_{i}(x_1)\label{eq:2.2}
\end{eqnarray}
for some nonnegative integer $k$. }
\end{definition}

Let $W$ be a vector space as before. Let $U$ be any $S_{trig}$-local subset of $\mathcal{E}(W)$ and let $a(x),b(x)\in U$.
Notice that the relation (\ref{eq:2.2}) implies
 \begin{eqnarray}
 (x_{1}-x_2)^{k}a(x_{1})b(x_{2})\in \mbox{Hom}(W, W((x_{1},x_{2}))).     \label{eq:2.3}
 \end{eqnarray}
 Define $a(x)_{n}^{e}b(x)\in\mbox{(End W)}[[x,x^{-1}]]$ for $n\in \Z$
                  in terms of generating function
                 \begin{eqnarray}Y_{\mathcal{E}}^{e}(a(x),z)b(x)=
                  \sum_{n\in\mathbb{Z}}(a(x)_{n}^{e}b(x))z^{-n-1}   \label{eq:2.4}\end{eqnarray}
                  by
                 \begin{eqnarray}Y_{\mathcal{E}}^{e}(a(x),z)b(x)=
                  x^{-k}(e^{z}-1)^{-k}((x_1-x)^{k}a(x_{1})b(x))|_{x_{1}=xe^{z}},  \label{eq:2.5}
                  \end{eqnarray}
                 where $k$ is any nonnegative integer such that (\ref{eq:2.3}) holds and where
                 $(e^z-1)^{-k}$ stands for the inverse of $(e^z-1)^{k}$ in $\C((z))$.

Let $U$ be an $S_{trig}$-local subspace of $\mathcal{E}(W)$. We say
 $U$ is {\em $Y_{\mathcal{E}}^{e}$-closed} if
           \begin{eqnarray*}
                 a(x)_{n}^{e}b(x)\in U
                  \    \    \mbox{ for all }a(x),b(x)\in U,\ n\in\mathbb{Z}.
            \end{eqnarray*}

The following result was obtained in \cite{Li-phi} (Theorem 5.4):

\begin{thm}\label{thm2.2}
Let $W$ be a vector space and let $U$ be any $S_{trig}$-local subset of $\mathcal{E}(W)$.
Then there exists a $Y_{\mathcal{E}}^{e}$-closed
$S_{trig}$-local subspace of $\mathcal{E}(W)$, which contains $U$ and $\textbf{1}_{W}$.
Denote by $\langle U\rangle_{e}$ the smallest such subspace. Then
$(\langle U\rangle_{e},Y_{\mathcal{E}}^{e},\textbf{1}_{W})$ carries the structure of a weak quantum vertex algebra
and  $W$ is a $\phi$-coordinated $\langle U\rangle_{e}$-module with
 $Y_{W}(a(x),z) = a(z)$  for $a(x)\in\langle U\rangle_{e}.$
\end{thm}

\section{Algebra $\A(h)$ and weak quantum vertex algebra $V_{\A(h)}$}

In this section, we first introduce an associative algebra $\A(h)$ associated to a formal power series $h(z)$ satisfying a certain condition, and then associate a weak quantum vertex algebra $V_{\A(h)}$ to this associative algebra and
establish an isomorphism between the category of suitably defined restricted $\A(h)$-modules and that of
$V_{\A(h)}$-modules.

Let $h(z)\in {\mathbb{C}}[[z]]$ such that $h(z)h(-z)=1$, which is fixed throughout this section.
Notice that we have $h(0)^{2}=1$. That is, $h(0)=\pm 1$.

\begin{definition}\label{def}
{\em Define $\A(h)$ to be the associative algebra
with identity over $\mathbb{C}$ with generators
$$e_{n},\   f_{n},\    \psi_{n}\   \   (n\in \mathbb{Z}),$$
 subject to relations
\begin{eqnarray}
&&e(z)e(w)=h(w-z)e(w)e(z), \    \   \    f(z)f(w)=h(z-w)f(w)f(z),\label{eqee}\\
&&\psi(z)e(w)=h(w-z)e(w)\psi(z), \    \   \psi(z)f(w)=h(z-w)f(w)\psi(z),\label{eqpsie}\\
&&[e(z),f(w)]=z^{-1}\delta\left(\frac{w}{z}\right)\psi(z)\label{eqef}\\
&&[\psi(z),\psi(w)]=0,\label{eqpsipsi}
\end{eqnarray}
where $a(x)=\sum_{n\in \mathbb{Z}}a_{n}x^{-n-1}$ for $a=e,f,\psi$.} 
\end{definition}

\begin{rmk}\label{inverse-relation}
{\em Note that since $h(0)=\pm 1$, $h(z-w)$ is an invertible element of $\C[[z,w]]$. Then
the relations (\ref{eqpsie}) are equivalent to
\begin{eqnarray}
e(w)\psi(z)=h(w-z)^{-1}\psi(z)e(w), \    \   \  f(w)\psi(z)=h(z-w)^{-1}\psi(z)f(w).\label{eq:3.2-new}
\end{eqnarray}}
\end{rmk}

\begin{rmk}\label{topological}
{\em Note that the commutation relations in the definition  involve infinite sums. For example,  we have
 \begin{eqnarray}
 e_{m}e_{n}=\sum_{r,i\ge 0}h_{r}\binom{r}{i}(-1)^{i}e_{n+r-i}e_{m+i}
 \end{eqnarray}
for $m,n\in \Z$, where
$$h(z)=\sum_{r\ge 0}h_{r}z^{r}.$$
In view of this, $\A(h)$ is a topological algebra in nature.}
\end{rmk}

Let $U$ denote a three-dimensional vector space with a designated basis $\{ e,f,\psi\}$. For convenience, set
$${\mathcal{B}}=\{ e,f,\psi\}\subset U.$$
Using the fact that
$$\left(\frac{\partial}{\partial z}+\frac{\partial}{\partial w}\right)h(z-w)=0,$$
by a straightforward argument we have:

\begin{lem}
 The algebra $\A(h)$ admits a derivation $d$ such that
 \begin{eqnarray}
 d(a_{n})=-na_{n-1}\   \   \   \mbox{ for }a\in U,\  n\in \Z,
 \end{eqnarray}
which  amounts to that $d(a(x))=\frac{d}{dx}a(x)$ for $a\in U$.
 \end{lem}

 \begin{definition}
{\em We define a  {\em restricted} $\A(h)$-module to be an $\A(h)$-module such that for any $a\in U,\ w\in W$,
$a_{n}w=0$ for $n$ sufficiently large, and  $W$ with  the discrete topology is a continuous module.
A nonzero vector $v$ in an $\A(h)$-module is called a {\em vacuum vector} if $a_{n}v=0$ for $a\in U, \ n\geq0$.
Furthermore, a {\em vacuum  $\A(h)$-module} is an $\A(h)$-module $W$ together with a vacuum vector
which generates $W$ as an $\A(h)$-module.}
\end{definition}

We have the following facts about a general vacuum  $\A(h)$-module:

\begin{lem}\label{span1}
Let $W$ be a vacuum $\A(h)$-module with vacuum vector $v$.
Set $F_{0}=\mathbb{C}v$ and $F_{k}=0$ for $k<0$.
For any positive integer $k$, define $F_{k}$ to be the linear span of vectors
$$a_{1}(-m_{1})\cdots a_{r}(-m_{r})v$$
for $r\geq1, a_{1},\dots, a_{r}\in \B,\  m_{1},\dots,m_{r}\geq 1$ with $m_{1}+\cdots+m_{r}\leq k.$
Then the subspaces $F_{k}$ for $k\in \mathbb{Z}$ form an ascending  filtration of $W$.
Furthermore, we have
\begin{eqnarray}\label{eq:3.1}
a(m)F_{k}\subseteq F_{k-m}\   \   \   \mbox{ for }a\in U,\  m,k\in \mathbb{Z}.
\end{eqnarray}
\end{lem}

\begin{proof} We first prove (\ref{eq:3.1}).
By definition, (\ref{eq:3.1}) is always true  for $m<0$. Now we consider $m\geq 0$.
If $k\le 0$, it is true as $F_{k}=0$ for $k<0$ and $F_{0}=\mathbb{C}v$.
Assume $k\geq 1$.
Note that from the defining relations of $\A(h)$ and Remark \ref{inverse-relation}, for $a,b\in \B,\ m,n\in \Z$,  we have
\begin{eqnarray}
a(m)b(n)=g(0) b(n)a(m)+\sum_{i,j\geq 0, i+j\geq 1}\alpha_{ij}b(n+i)a(m+j)+\beta \psi(m+n)\label{eq:3.2}
\end{eqnarray}
for some $\alpha_{ij},\beta\in {\mathbb{C}}$, depending on $a,b$.
Then (\ref{eq:3.1}) follows from this and an induction on $k$.
 From (\ref{eq:3.1}), we see that $\cup_{k\geq0}F_{k}$ is a submodule of $W$,
which contains the generator $v$ of $W$. Consequently, $W=\cup_{k\geq 0}F_{k}$.
Thus the subspaces $F_{k}$ for $k\in \mathbb{Z}$ form an ascending filtration of $W$.
\end{proof}

On the other hand, we have:

\begin{lem}\label{span2}
Let $W$ be a vacuum $\A(h)$-module with vacuum vector $v$. For  $n\in \N$, set
$$E_{n}={\rm span}\{a^{(1)}(m_{1})\cdots a^{(r)}(m_{r})v\  |\   0\leq r\leq n, \  a^{(i)}\in \B, \  m_{i}\in \mathbb{Z}\}.$$
Then $E_{n}$ for $n\in \N$ form an ascending filtration of $W$.
Furthermore, if $g(0)=-1$, then for every $n\in \N$,
$E_{n}$ is linearly spanned by vectors
\begin{eqnarray}
e(-m_{1})\cdots e(-m_{r})f(-n_{1})\cdots f(-n_{s})\psi(-k_{1})\cdots \psi(-k_{l})v\label{eq:3.3}
\end{eqnarray}
for $r,s,l\geq 0,\  m_{i},n_{j},k_{t}\ge 1$ with
$$r+s+l\leq n,\  \   m_{1}>\cdots>m_{r},\  \   n_{1}>\cdots>n_{s}, \   \   k_{1}\geq\cdots\geq k_{l}.$$
If $g(0)=1$, then for every  $n\in \N$, $E_{n}$ is linearly spanned by the vectors
\begin{eqnarray}
e(-m_{1})\cdots e(-m_{r})f(-n_{1})\cdots f(-n_{s})\psi(-k_{1})\cdots \psi(-k_{l})v,  \label{eq:3.4.}
\end{eqnarray}
where  $r,s,l\geq 0, \   m_{i},n_{j},k_{t}\ge 1$ with
$$r+s+l\leq n,\   \   m_{1}\geq\cdots\geq m_{r},\  \    n_{1}\geq\cdots\geq n_{s}, \   \   k_{1}\geq\cdots\geq k_{l}.$$
\end{lem}

\begin{proof} It is clear that $E_{n}$ for $n\in \N$ form an ascending filtration for $W$.
 For $n\in \N$, set
$$E_{n}^{'}={\rm span}\{u_{1}(-m_{1})\cdots u_{r}(-m_{r})v\  |\   0\leq r\leq n,\  u_{1},\dots, u_{r}\in \B, \  m_{1},\dots,m_{r}\geq 1\}.$$
Now, we prove $E_{n}'=E_{n}$. From definition, we have $E_{n}^{'}\subseteq E_{n}$ for $n\in\N$.
For $u\in U,\   m\in \Z$ with $m<0$, we have $u(m)E_{n}^{'}\subseteq E_{n+1}^{'}$ by definition.
 On the other hand, by using induction on $n$ and (\ref{eq:3.2}),
we get $u(m)E_{n}^{'}\subseteq E_{n}^{'}$  for $u\in \B,\ m\geq 0$.  Note that $E_{0}'=\C v$.
It then follows from an induction that $E_{n}\subseteq E_{n}^{'}$ for $n\in\N$. Thus $E_n=E_{n}^{'}$ for $n\in\N$.
Now, let $a,b\in \B,\ m,n\in \Z$ with $m,n\le -1$ and let $w\in E_{r}\cap F_{s}$ with $r,s\in \Z$
(where $F_s$ is defined in Lemma \ref{span1}).
Notice that $a(m)b(n)w\in E_{r+2}\cap F_{s-m-n}$. From (\ref{eq:3.2}) we have
\begin{eqnarray}
a(m)b(n)w\equiv g(0) b(n)a(m)w\  \mod (E_{r+2}\cap F_{s-m-n-1}+E_{r+1}\cap F_{s-m-n}).
\end{eqnarray}
Then (\ref{eq:3.3}) and (\ref{eq:3.4.}) follow from induction on $(r,s)$ with respect to the lexicographical order.
\end{proof}

Next,  we give a tautological construction of  a universal vacuum $\A(h)$-module.
Set $\mathfrak{T}=T(U\otimes\mathbb{C}[t,t^{-1}])$, the tensor algebra over vector space $U\otimes\mathbb{C}[t,t^{-1}]$.
Let $d$ be the derivation of algebra $\mathfrak{T}$ determined by
$$d(u\otimes t^{n})=-n(u\otimes t^{n-1})\  \    \   \mbox{ for }u\in \B, \   n\in\Z.$$
Set $\mathfrak{T}_{-}=\sum_{n\ge 1}\mathfrak{T}_{-n}$.
Furthermore, set $J=\mathfrak{T}\C[d]\mathfrak{T}_{-}$, a left ideal of $\mathfrak{T}$.
Then set $V_{\mathfrak{T}}=\mathfrak{T}/J$, a left $\mathfrak{T}$-module.
We see that for any $w\in V_{\mathfrak{T}}$ and $u\in \B,\   u(n)w=0$ for $n$ sufficiently large
(as for any $a\in \mathfrak{T}$, $u(n)a\in \mathfrak{T}_{-}\subset J$ for $n$ sufficiently large).
Then define $V_{\A(h)}$ to be the quotient $\mathfrak{T}$-module of $V_{\mathfrak{T}}$ modulo the submodule
corresponding to the defining relations of $\A(h)$.
Let $\mathbf{1}$ denote the image of $1$ in $V_{\A(h)}$. Then $(V_{\A(h)},\mathbf{1})$ is a vacuum $\A(h)$-module.
Since $d J\subset J$, $d$ acts on $V_{\A(h)}$ such that $d\mathbf{1}=0,\ \ [d,u(x)]=\frac{d}{dx}u(x)$ for $u\in \B$.
We see that vacuum module $V_{\A(h)}$ is universal in the sense that for any vacuum $\A(h)$-module $(W, w_{0})$
on which $d$ acts such that $dw_{0}=0$, $[d,u(x)]=\frac{d}{dx}u(x)$ for $u\in \B$,
there exists a unique $\A(h)$-module homomorphism from $V_{\A(h)}$ to $W$, sending $\mathbf{1}$ to $w_{0}$.

Here, we have:

\begin{thm}
There exists a weak quantum vertex algebra structure on the vacuum $\A(h)$-module $(V_{\A(h)},\bf 1)$,
which is uniquely determined by the condition that $\bf{1}$ is the vacuum vector and
$$Y(e(-1){\bf 1},x)=e(x),\ \   Y(f(-1){\bf 1},x)=f(x),\  \  Y(\psi(-1){\bf 1},x)=\psi(x).$$
Furthermore, for every restricted $\A(h)$-module $W$, there exists a $V_{\A(h)}$-module structure
$Y_{W}(\cdot,x)$ on $W$,
which is uniquely determined by
$$Y_{W}(e(-1){\bf 1},x)=e(x),\   \   Y_{W}(f(-1){\bf 1},x)=f(x),\  \  Y_{W}(\psi(-1){\bf 1},x)=\psi(x).$$
\end{thm}

\begin{proof}
 Let $W$ be any restricted $\A(h)$-module. Then  the direct sum $V_{\A(h)}\oplus W$, denoted by $\bar{W}$, is also a restricted $\A(h)$-module.
 Set $U_{\bar{W}}=\{e(x), f(x), \psi(x)\}\cup \{1_{\bar{W}}\}\subseteq \E(\bar{W})$.
 From the defining relations of $\A(h)$, we see that $U_{\bar{W}}$ is an ${\mathcal{S}}$-local subset of $\E(\bar{W})$.
 Then by \cite{Li-nonlocal} (Theorem 5.8), $U_{\bar{W}}$ generates a weak quantum vertex algebra $V_{\bar{W}}$
 under the vertex operator operation $Y_{\E}$, and $\bar{W}$ is a faithful $V_{\bar{W}}$-module with $Y_{\bar{W}}(u(x),x_{0})=u(x_{0})$ for $u(x)\in V_{\bar{W}}$. From  \cite{Li-nonlocal} (Proposition 6.7),
 we see that $V_{\bar{W}}$ is a vacuum $A(h)$-module with $e(z), \ f(z),\ \psi(z)$ acting as
 $Y_{\E}(e(x),z),\ Y_{\E}(f(x),z), \ Y_{\E}(\psi(x),z)$, respectively.
  Since $V_{\A(h)}$ is universal, there exists an $\A(h)$-module homomorphism $\varphi$ from $V_{\A(h)}$ to $V_{\bar{W}}$, sending $\bf 1$ to $1_{\bar{W}}$. Since $V_{\A(h)}$ is a vacuum $\A(h)$-module, and
  $V_{\A(h)}$ admits an action of $d$ such that $d\mathbf{1}=0, \ [d,u(x)]=\frac{d}{dx}u(x)$ for $u\in U$.
  Then by  \cite{Li-nonlocal} (Theorem 6.3), there exists  a weak quantum vertex algebra structure on $V_{\A(h)}$
  with $\bf 1$ as the vacuum vector such that $Y(e,x)=e(x),\ Y(f,x)=f(x),\ Y(\psi,x)=\psi(x).$
  Furthermore, from Theorem 6.5 in \cite{Li-nonlocal}, $\bar{W}$ is a $V_{\A(h)}$-module with $W$ as a submodule.
  \end{proof}

\section{Non-degeneracy of weak quantum vertex algebra $V_{\A(h)}$}
In this section,  we restrict ourselves to the case where  $h(z)$ is factorizable in the sense that
$h(z)=\pm q(-z)q(z)^{-1}$ for some nonzero $q(z)\in \C[[z]]$. In this case, we
determine a P-B-W type basis for the vacuum $\A(h)$-module $V_{\A(h)}$ and
prove that the weak quantum vertex algebra $V_{\A(h)}$ is a non-degenerate quantum vertex algebra.

Throughout this section,  we assume  that $h(z)=\pm q(-z)q(z)^{-1}$ where  $q(z)\in \mathbb{C}[[z]]$ such that $q(0)=1$.
To obtain a basis of the vacuum $\A(h)$-module $V_{\A(h)}$, we shall use a vertex algebra to obtain
a vacuum $\A(h)$-module for $h(z)=q(-z)q(z)^{-1}$, whereas for $h(z)=-q(-z)q(z)^{-1}$, we shall use a vertex superalgebra.

First, we consider the case $h(z)=q(-z)q(z)^{-1}$.
Let $\mathfrak{g}=\C\bar{e}\oplus \C\bar{f}\oplus \C \bar{\psi}$ be a (Heisenberg) Lie algebra with bracket relations
\begin{eqnarray}
[\bar{\psi},\bar{e}]=0=[\bar{\psi},\bar{f}],\    \    \    \   [\bar{e},\bar{f}]=\bar{\psi}.
\end{eqnarray}
Then we have the loop Lie algebra $L({\mathfrak{g}})=\mathfrak{g}\otimes\C[t,t^{-1}]$.

Follow a common practice alternatively to denote $a\otimes t^{m}$ by $a(m)$ for $a\in \g,\ m\in \Z$.
For $u\in \g$, form a generating function
$$u(x)=\sum_{n\in \Z}u(n)x^{-n-1}.$$
Let $\mathfrak{g}\otimes\C[t]$ act trivially on $\C$.
Form an induced module
$$V_{L(\mathfrak{g})}=U(L({\mathfrak{g}}))\otimes_{U(\mathfrak{g\otimes\C[t]})}\C.$$
 Set $\mathbf{1}=1\otimes1\in V_{L(\mathfrak{g})}$. In view of the P-B-W theorem,
 $V_{L(\mathfrak{g})}$ has a basis consisting of vectors
$$\bar{e}(-m_{1})\cdots\bar{e}(-m_{r})\bar{f}(-n_{1})\cdots\bar{f}(-n_{s})\bar{\psi}(-k_{1})\cdots\bar{\psi}(-k_{l})\mathbf{1}$$
for $r,s,l\geq 0$, $m_{1}\geq\cdots\geq m_{r}\geq 1, n_{1}\geq\cdots\geq n_{s}\geq 1, k_{1}\geq\cdots\geq k_{l}\geq 1$.

Identify $\mathfrak{g}$ as a subspace of $V_{L(\mathfrak{g})}$ through the linear map
$u\in \mathfrak{g}\mapsto u(-1)\mathbf{1}$.
It is known that there exists a vertex algebra structure on $V_{L(\mathfrak{g})}$, which is uniquely determined
by the conditions that $\mathbf{1}$ is the vacuum vector and that $Y(u,x)=u(x)$ for $u\in \mathfrak{g}$.

Note that $L(\mathfrak{g})$ is a $\Z$-graded Lie algebra with
$$ \deg(a\otimes t^{n})=-n\   \   \mbox{ for }a\in\mathfrak{g},\ n \in \Z.$$
Naturally,  $U(L(\g))$ is a $\Z$-graded algebra. It follows that
$V_{L(\g)}$ is a $\Z$-graded $L(\g)$-module with $(V_{L(\g)})_{(n)}=0$ for $n<0$,
$(V_{L(\g)})_{(0)}=\C \mathbf{1}$, and $(V_{L({\g})})_{(1)}=\g$.
Furthermore, we have
$$a(m)(V_{L(\g)})_{(n)}\subset (V_{L(\g)})_{(n-m)}\    \    \   \mbox{  for }a\in\mathfrak{g},\   m,n\in\Z.$$

Next, we are going to define a vacuum $\A(h)$-module structure on the vertex algebra $V_{L(\g)}$.
First, we have:

\begin{lem}\label{pseudo}
For any $p_1(t),p_2(t)\in \C[[t]]$, there exists
$$\phi(t)\in \Hom_{\C}(V_{L(\g)}, V_{L(\g)}\otimes\C[[t]])\subset \Hom (V_{L(\g)}, V_{L(\g)}[[t]]),$$
which is uniquely determined by
\begin{eqnarray*}
&&\phi(t)(\mathbf{1})=\mathbf{1},\   \     \phi(t)Y(v,x)=Y(\phi(t-x)v,x)\phi(t) \   \   \mbox{for} \ v\in V_{L(\g)},\\
&&\phi(t)\bar{e}(x)=p_1(t-x)\bar{e}(x)\phi(t), \    \    \phi(t)\bar{f}(x)=p_2(t-x)\bar{f}(x)\phi(t),\nonumber\\
&&\phi(t)\bar{\psi}(x)=p_1(t-x)p_2(t-x)\bar{\psi}(x)\phi(t).
\end{eqnarray*}
Furthermore, we have
\begin{eqnarray}
\phi(x)\phi(t)=\phi(t)\phi(x),\   \   \  [\mathcal{D}, \phi(t)]=\frac{d}{dt}\phi(t),
\end{eqnarray}
where $\D$ is the linear operator on $V_{L(\g)}$, defined by $\D v=v_{-2}{\bf 1}$ for $v\in V_{L(\g)}$.
\end{lem}


\begin{proof} Note that $\C[[t]]$ is a vertex algebra with $1$ as the vacuum vector and
$$Y(p(t),x)q(t)=(e^{-x\frac{d}{dt}}p(t))q(t)=p(t-x)q(t)$$
for $p(t), q(t)\in \C[[t]]$. Then we have  a tensor product vertex algebra $V_{L(\g)}\otimes\C[[t]]$, whose
vertex operator map, denoted by $\bar{Y}(\cdot,x)$, is given by
$$\bar{Y}(u\otimes p(t),x)=Y(u,x)\otimes p(t-x)$$
for $u\in V,\  p(t)\in\C[[t]]$.
Then we have
$$[\bar{Y}(\bar{e}\otimes p_1(t), x_{1}), \bar{Y}(\bar{e}\otimes p_1(t), x_{2})]=[Y(\bar{e},x_{1}), Y(\bar{e},x_2)]\otimes p_1(t-x_{1})p_1(t-x_{2})=0,$$
$$[\bar{Y}(\bar{f}\otimes p_2(t), x_{1}), \bar{Y}(\bar{f}\otimes p_2(t), x_{2})]=[Y(\bar{f},x_{1}), Y(\bar{f},x_2)]\otimes p_2(t-x_{1})p_2(t-x_{2})=0,$$
$$[\bar{Y}(\bar{\psi}\otimes p_1(t)p_2(t), x_{1}), \bar{Y}(\bar{e}\otimes p_1(t), x_{2})]
=[Y(\bar{\psi},x_{1}), Y(\bar{e},x_2)]\otimes p_1(t-x_{1})p_2(t-x_{1})p_1(t-x_{2})=0,$$
\begin{eqnarray*}
&&[\bar{Y}(\bar{\psi}\otimes p_1(t)p_2(t), x_{1}), \bar{Y}(\bar{f}\otimes p_2(t), x_{2})]\\
&=&[Y(\bar{\psi},x_{1}), Y(\bar{f},x_2)]\otimes p_1(t-x_{1})p_2(t-x_{1})p_2(t-x_{2})=0,
\end{eqnarray*}
\begin{eqnarray*}
&&[\bar{Y}(\bar{\psi}\otimes p_1(t)p_2(t), x_{1}), \bar{Y}(\bar{\psi}\otimes p_1(t)p_2(t), x_{2})]\\
&=&[Y(\bar{\psi},x_{1}), Y(\bar{\psi},x_2)]\otimes p_1(t-x_{1})p_2(t-x_{1})p_1(t-x_{2})p_2(t-x_{2})\\
&=&0,
\end{eqnarray*}
\begin{eqnarray*}
&&[\bar{Y}(\bar{e}\otimes p_1(t), x_{1}), \bar{Y}(\bar{f}\otimes p_2(t), x_{2})]\nonumber\\
&=&[Y(\bar{e},x_{1}), Y(\bar{f},x_2)]\otimes p_1(t-x_{1})p_2(t-x_{2})\\
&=&x_{1}^{-1}\delta\left(\frac{x_{2}}{x_{1}}\right)Y(\bar{\psi},x_2)\otimes p_1(t-x_{1})p_2(t-x_{2})\\
&=&x_{1}^{-1}\delta\left(\frac{x_{2}}{x_{1}}\right)Y(\bar{\psi},x_2)\otimes p_1(t-x_{2})p_2(t-x_{2})\\
&=&x_{1}^{-1}\delta\left(\frac{x_{2}}{x_{1}}\right)\bar{Y}(\bar{\psi}\otimes p_1(t)p_2(t),x_2).
\end{eqnarray*}
In view of this, $V_{L(\g)}\otimes \C[[t]]$ is an $L(\g)$-module
with $\bar{e}(x)$, $\bar{f}(x)$, $\bar{\psi}(x)$ acting as
$$\bar{Y}(\bar{e}\otimes  p_1(t),x),\   \   \  \bar{Y}(\bar{f}\otimes  p_2(t),x),\   \   \
\bar{Y}(\bar{\psi}\otimes  p_1(t)p_2(t),x),$$
respectively.
 It is clear that ${\bf 1}\otimes 1$ is a vacuum vector.
 It follows that there exists a (unique) $L(\g)$-module homomorphism $\phi$ from
 $V_{L(\g)}$ to $V_{L(\g)}\otimes\C[[t]]$ with $\phi({\bf 1})={\bf 1}\otimes 1$.
 We have
$$\phi(\bar{e})=\bar{e}\otimes p_1(t),\   \   \phi(\bar{f})=\bar{f}\otimes p_2(t),\   \
\phi(\bar{\psi})=\bar{\psi}\otimes p_1(t)p_2(t).$$
 Thus
 $$\phi( Y(a,x)v)=\bar{Y}(\phi(a),x)\phi(v)\   \   \  \mbox{ for }a\in \g,\  v\in V_{L(\g)}.$$
As $\g$ generates $V_{L(\g)}$ as a vertex algebra, it follows that $\phi$ is a vertex algebra homomorphism  from
$V_{L(\g)}$ to $V_{L(\g)}\otimes\C[[t]]$.
Write $\phi$ as $\phi(t)$ with a visible dependence on $t$. Then
$$\phi(t)(\mathbf{1})=\mathbf{1},\  \   \phi(t)(\bar{e})=\bar{e}\otimes p_1(t), \   \   \phi(t)(\bar{f})=\bar{f}\otimes p_2(t),
\   \   \phi(t)(\bar{\psi})=\bar{\psi}\otimes p_1(t)p_2(t)$$
and
\begin{eqnarray*}
\phi(t)(Y(u,x)v)
=\bar{Y}(\phi(t)(u),x)\phi(t)(v)
=Y(\phi(t-x)u,x)\phi(t)v
\end{eqnarray*}
for $u,v\in V_{L(\g)}$. Then the first three relations in the furthermore assertion follows immediately.

It is clear that $\phi(x)\phi(t)=\phi(t)\phi(x)$ on $\g$. Then it follows that $\phi(x)\phi(t)=\phi(t)\phi(x)$ on $V_{L(\g)}$
as $\g$ generates $V_{L(\g)}$ as a vertex algebra.
Noticing that the $\D$-operator of $V_{L(\g)}\otimes \C[[t]]$ is $\D\otimes 1+1\otimes (-d/dt)$, as $\phi$ is a vertex algebra homomorphism, we have
$\phi \D =(\D\otimes 1+1\otimes (-d/dt))\phi$, which implies that
$[\D,\phi(t)]=\frac{d}{dt}\phi(t)$.
\end{proof}

Recall that $h(z)=q(-z)q(z)^{-1}$, where $q(t)\in \C[[t]]$ with $q(0)=1$.
As an immediate consequence of Lemma \ref{pseudo} we have:

\begin{cor}\label{relation}
There exists an invertible element
$$\phi(t)\in \Hom_{\C}(V_{L(\g)}, V_{L(\g)}\otimes \C[[t]])\subset (\End_{\C} V_{L(\g)})[[t]]$$
such that
\begin{eqnarray*}
&&\phi(t)(\mathbf{1})=\mathbf{1},\   \    \nonumber\\
&&\phi(t)\bar{e}(x)=q(t-x)\bar{e}(x)\phi(t), \    \    \phi(t)\bar{f}(x)=q(x-t)\bar{f}(x)\phi(t),\nonumber\\
&& \phi(t)\bar{\psi}(x)=q(t-x)q(x-t)\bar{\psi}(x)\phi(t), \nonumber\\
&& [\mathcal{D}, \phi(t)]=\frac{d}{dt}\phi(t),\nonumber\\
&& \phi(x)\phi(t)=\phi(t)\phi(x).
\end{eqnarray*}
\end{cor}

View $V_{L(\g)}\otimes \C[[t]]$ naturally as a subspace of $V_{L(\g)}[[t]]$.
In this way, we view $\phi(t)$ as a linear map from $V_{L(\g)}$ to $V_{L(\g)}[[t]]$.

\begin{prop}\label{map}
Let $\phi(t)$ be the element of $(\End_{\C}V_{L(\g)})[[t]]$ obtained in Corollary \ref{relation}.
 Then the  assignment
$$e(x)=\bar{e}(x)\phi(x),\  \    f(x)=\bar{f}(x)\phi(x),\  \    \psi(x)=\bar{\psi}(x)\phi(x)\phi(x)$$
uniquely defines a vacuum $\A(h)$-module structure on $V_{L(\g)}$ with $\mathbf{1}$ as the generator, such that  \begin{eqnarray}
[\D,e(x)]=\frac{d}{dx}e(x),\ [\D,f(x)]=\frac{d}{dx}f(x),\ [\D,\psi(x)]=\frac{d}{dx}\psi(x),\label{eq:3.6.}
\end{eqnarray}
where $\D$ is the linear operator on $V_{L(\g)}$, defined by $\D v=v_{-2}{\bf 1}$ for $v\in V_{L(\g)}$.
\end{prop}

\begin{proof} By Corollary \ref{relation}, we have
\begin{eqnarray*}
e(x_{1})e(x_{2})&=&\bar{e}(x_{1})\phi(x_{1})\bar{e}(x_{2})\phi(x_{2})\\
&=&\bar{e}(x_{1})\bar{e}(x_{2})q(x_{1}-x_{2})\phi(x_{1})\phi(x_{2})\\
&=&\bar{e}(x_{2})\bar{e}(x_{1})\phi(x_{2})\phi(x_{1})q(x_{1}-x_{2})\\
&=&\bar{e}(x_{2})\phi(x_{2})\bar{e}(x_{1})\phi(x_{1})q(x_{1}-x_{2})q(x_{2}-x_{1})^{-1}\\
&=&e(x_{2})e(x_{1})g(x_{2}-x_{1}),
\end{eqnarray*}
\begin{eqnarray*}
f(x_{1})f(x_{2})&=&\bar{f}(x_{1})\phi(x_{1})\bar{f}(x_{2})\phi(x_{2})\\
&=&\bar{f}(x_{1})\bar{f}(x_{2})q(x_{2}-x_{1})\phi(x_{1})\phi(x_{2})\\
&=&\bar{f}(x_{2})\bar{f}(x_{1})\phi(x_{2})\phi(x_{1})q(x_{2}-x_{1})\\
&=&\bar{f}(x_{2})\phi(x_{2})\bar{f}(x_{1})\phi(x_{1})q(x_{1}-x_{2})^{-1}q(x_{2}-x_{1})\\
&=&f(x_{2})f(x_{1})g(x_{1}-x_{2}),
\end{eqnarray*}
\begin{eqnarray*}
&&[e(x_{1}), f(x_{2})]\\
&=&\bar{e}(x_{1})\phi(x_{1})\bar{f}(x_{2})\phi(x_{2})-\bar{f}(x_{2})\phi(x_{2})\bar{e}(x_{1})\phi(x_{1})\\
&=&\bar{e}(x_{1})\bar{f}(x_{2})q(x_{2}-x_{1})\phi(x_{1})\phi(x_{2})-\bar{f}(x_{2})\bar{e}(x_{1})q(x_{2}-x_{1})\phi(x_{1})\phi(x_{2})\\
&=&q(x_{2}-x_{1})[\bar{e}(x_{1}),\bar{f}(x_{2})]\phi(x_{1})\phi(x_{2})\\
&=&q(x_{2}-x_{1})x_{1}^{-1}\delta\left(\frac{x_{2}}{x_{1}}\right)\bar{\psi}(x_{2})\phi(x_{1})\phi(x_{2})\\
&=&q(0)x_{1}^{-1}\delta\left(\frac{x_{2}}{x_{1}}\right)\bar{\psi}(x_{2})\phi(x_{2})\phi(x_{2})\\
&=&x_{1}^{-1}\delta\left(\frac{x_{2}}{x_{1}}\right)\psi(x_{2}),
\end{eqnarray*}
\begin{eqnarray*}
&&[\psi(x_{1}), \psi(x_{2})]\\
&=&\bar{\psi}(x_{1})\phi(x_{1})\phi(x_{1})\bar{\psi}(x_{2})\phi(x_{2})\phi(x_{2})-\bar{\psi}(x_{2})\phi(x_{2})\phi(x_{2})\bar{\psi}(x_{1})\phi(x_{1})\phi(x_{1})\\
&=&q(x_{1}-x_{2})^{2}q(x_{2}-x_{1})^{2}\bar{\psi}(x_{1})\bar{\psi}(x_{2})\phi(x_{1})\phi(x_{1})\phi(x_{2})\phi(x_{2})\\
&-&q(x_{2}-x_{1})^{2}q(x_{1}-x_{2})^{2}\bar{\psi}(x_{2})\bar{\psi}(x_{1})\phi(x_{2})\phi(x_{2})\phi(x_{1})\phi(x_{1})\\
&=&q(x_{1}-x_{2})^{2}q(x_{2}-x_{1})^{2}[\bar{\psi}(x_{1}),\bar{\psi}(x_{2})]\phi(x_{1})\phi(x_{1})\phi(x_{2})\phi(x_{2})\\
&=&0,
\end{eqnarray*}
\begin{eqnarray*}
&&\psi(x_{1})e(x_{2})\\
&=&\bar{\psi}(x_{1})\phi(x_{1})\phi(x_{1})\bar{e}(x_{2})\phi(x_{2})\\
&=&q(x_{1}-x_{2})^{2}\bar{\psi}(x_{1})\bar{e}(x_{2})\phi(x_{1})\phi(x_{1})\phi(x_{2})\\
&=&q(x_{1}-x_{2})^{2}\bar{e}(x_{2})\bar{\psi}(x_{1})\phi(x_{1})\phi(x_{1})\phi(x_{2})\\
&=&q(x_{1}-x_{2})^{2}q(x_{1}-x_{2})^{-1}q(x_{2}-x_{1})^{-1}\bar{e}(x_{2})\phi(x_{2})\bar{\psi}(x_{1})\phi(x_{1})\phi(x_{1})\\
&=&q(x_1-x_2)q(x_2-x_1)^{-1}e(x_2)\psi(x_1)\\
&=&g(x_{2}-x_{1})e(x_2)\psi(x_1),
\end{eqnarray*}
\begin{eqnarray*}
&&\psi(x_{1})f(x_{2})\\
&=&\bar{\psi}(x_{1})\phi(x_{1})\phi(x_{1})\bar{f}(x_{2})\phi(x_{2})\\
&=&q(x_{2}-x_{1})^{2}\bar{\psi}(x_{1})\bar{f}(x_{2})\phi(x_{1})\phi(x_{1})\phi(x_{2})\\
&=&q(x_{2}-x_{1})^{2}\bar{f}(x_{2})\bar{\psi}(x_{1})\phi(x_{2})\phi(x_{1})\phi(x_{1})\\
&=&q(x_{2}-x_{1})^{2}q(x_{2}-x_{1})^{-1}q(x_{1}-x_{2})^{-1}\bar{f}(x_{2})\phi(x_{2})\bar{\psi}(x_{1})\phi(x_{1})\phi(x_{1})\\
&=&q(x_2-x_1)q(x_1-x_2)^{-1}f(x_2)\psi(x_1)\\
&=&g(x_{1}-x_{2})f(x_2)\psi(x_1).
\end{eqnarray*}
This shows that $V_{L(\g)}$ is an $\A(h)$-module. Obviously, $\mathbf{1}$ is a vacuum vector.

Next, we prove that $V_{L(\g)}$ is generated by $\bf{1}$ as an $\A(h)$-module.
Let $W$ be the $\A(h)$-submodule of $V_{L(\g)}$ generated by $\bf{1}$.
From Corollary \ref{relation}, we have
\begin{eqnarray*}
&&\phi^{-1}(x_{1})e(x)=q(x_1-x)^{-1}e(x)\phi^{-1}(x_{1}),\\
&&\phi^{-1}(x_{1})f(x)=q(x-x_1)^{-1}f(x)\phi^{-1}(x_{1}),\\
&&\phi^{-1}(x_{1})\psi(x)=q(x_1-x)^{-1}q(x-x_1)^{-1}\psi(x)\phi^{-1}(x_{1}).
\end{eqnarray*}
With $\phi(x)^{-1}\mathbf{1}=\bf{1}$, it then follows that $\phi(x)^{-1}W\subset W((x))$.
Since $$\bar{e}(x)=e(x)\phi(x)^{-1},\  \   \bar{f}(x)=f(x)\phi(x)^{-1},\   \   \bar{\psi}(x)=\psi(x)\phi(x)^{-1}\phi(x)^{-1},$$
it follows that $W$ is an $L(\g)$-submodule of $V_{L(\g)}$.
Consequently, $W=V_{L(\g)}$. This implies that $V_{L(\g)}$ is a vacuum $\A(h)$-module.

On $V_{L(\g)}$, we have
$$[\D, \bar{e}(x)]=\frac{d}{dx}\bar{e}(x),\  \  [\D, \bar{f}(x)]=\frac{d}{dx}\bar{f}(x),\  \
[\D, \bar{\psi}(x)]=\frac{d}{dx}\bar{\psi}(x).$$
Then (\ref{eq:3.6.}) follows immediately.
\end{proof}

Furthermore, we have:

\begin{prop}\label{basis}
 View $V_{L(\g)}$ as a vacuum $\A(h)$-module. For $n\in \N$, set
$$E_n={\rm span}\{u^{(1)}(m_{1})\cdots u^{(r)}(m_{r})\mathbf{1}\ |\  0\leq r\leq n,\  u^{(i)}\in U, \ m_{i}\in \Z\}. $$
Then $E_n$ has a basis consisting of vectors
\begin{eqnarray}
e(-m_{1})\cdots e(-m_{r})f(-n_{1})\cdots f(-n_{s})\psi(-k_{1})\cdots \psi(-k_{l})\mathbf{1},\label{eq:3.5.}
\end{eqnarray}
  where $r,s,l\geq 0$ with $r+s+l\leq n$, and
   $$m_{1}\geq\cdots\geq m_{r}\geq 1,\   \    n_{1}\geq\cdots\geq n_{s}\geq 1, \   \   k_{1}\geq\cdots\geq k_{l}\geq 1.$$
\end{prop}

\begin{proof} From Lemma \ref{span2}, we see that $E_n$ is spanned by the vectors in (\ref{eq:3.5.}), so
 it remains to establish the linear independence.

 For convenience,  simply denote $V_{L(\g)}$ by $V$  in this proof.
 Recall that $V=\coprod_{n\in \N}V_{(n)}$ is an $\N$-graded $L(\g)$-module with $V_{(0)}=\C\mathbf{1}$ and $V_{(1)}=U$.
 For $n\in \Z$, set $\bar{F}_n=\oplus_{m\le n} V_{(m)}$. We have
 $\bar{F}_{n}=0$ for $n<0$, $\bar{F}_{0}=V_{(0)}=\C {\bf 1}$, and
\begin{eqnarray}\label{graded-module}
a(m)\bar{F}_{n}\subset \bar{F}_{n-m}\   \   \  \mbox{ for }a\in \{ \bar{e},\bar{f},\bar{\psi}\},\ m,n\in \Z.
\end{eqnarray}
 We see that the following vectors form a basis of $\bar{F}_n$:
 \begin{eqnarray}\label{ebasis-bar}
 \bar{e}(-m_{1})\cdots \bar{e}(-m_{r})\bar{f}(-n_{1})\cdots \bar{f}(-n_{s})\bar{\psi}(-k_{1})\cdots
 \bar{\psi}(-k_{l})\mathbf{1}\label{eq:3.7.}
\end{eqnarray}
for $r,s,l\geq 0$ with $\sum m_i+\sum n_j+\sum k_s\leq n$, and for  $m_i,n_j,k_t\in \Z$  with
$$m_{1}\geq\cdots\geq m_{r}\geq 1, \   \   n_{1}\geq\cdots\geq n_{s}\geq 1, \   \   k_{1}\geq\cdots\geq k_{l}\geq 1.$$
Write $\phi(t)=\sum_{i\geq 0}\phi_{i}t^{i}$ (with $\phi_i\in \End_{\C}V$) and $q(t)=\sum_{i\ge 0}q_it^i$ (with $q_i\in \C$).
We have
\begin{eqnarray*}
&&\phi_{i}\bar{e}(m)=\sum_{r\ge 0}\sum_{j\ge 0, r-j\le i} q_{r} \binom{r}{j}(-1)^{j}\bar{e}(m+j)\phi_{i-r+j},\\
&&\phi_{i}\bar{f}(m)=\sum_{r,j\ge 0,r-j\le i}q_{r}\binom{r}{j}(-1)^{r+j}\bar{f}(m+j)\phi_{i-r+j},\\
&&\phi_{i}\bar{\psi}(m)=\sum_{r,r',j,j'\ge 0, r+r'-j-j'\le i}q_rq_{r'}\binom{r}{j}\binom{r'}{j'}(-1)^{j+r'+j'}\bar{\psi}(m+j+j')\phi_{i-r+j-r'-j'}.
\end{eqnarray*}
 Notice that $\phi(t)\bar{F}_{0}\subset \bar{F}_{0}$ as $\phi(t){\bf 1}={\bf 1}$.
It then follows from induction and (\ref{graded-module}) that
\begin{eqnarray}
\phi_{i}\bar{F}_{n}\subset \bar{F}_{n}\   \   \   \mbox{ for }i\in \N,\ n\in \Z.
\end{eqnarray}
On the other hand, using the commutation relations above with $i=0$ (and induction) we get
\begin{eqnarray}
\phi_0 (w)\equiv w\ ({\rm mod} \bar{F}_{n-1})\   \   \  \mbox{  for }w\in \bar{F}_{n},\ n\in \Z,
\end{eqnarray}
noticing that $q_0=q(0)=1$.

From the relations $e(x)=\bar{e}(x)\phi(x),\  f(x)=\bar{f}(x)\phi(x), \  \psi(x)=\bar{\psi}(x)\phi(x)\phi(x)$,
  we have
  $$e(m)=\sum_{i\geq 0}\bar{e}(m+i)\phi_{i},\   \   \
  f(m)=\sum_{i\geq 0}\bar{f}(m+i)\phi_{i},$$
  $$\psi(m)=\sum_{i,j\ge 0}\bar{\psi}(i+j+m)\phi_{i}\phi_{j}$$
  for $m\in \Z$. It follows that $u(m)\bar{F}_{n}\subset \bar{F}_{n-m}$ for $u\in \B= \{e,f,\psi\},\  m,n\in \Z$.

 Now, define a linear endomorphism $\sigma$ of $V$, which sends the basis vector in  (\ref{ebasis-bar}) to
 the corresponding vector
  \begin{eqnarray*}
 e(-m_{1})\cdots e(-m_{r})f(-n_{1})\cdots f(-n_{s})\psi(-k_{1})\cdots
 \psi(-k_{l})\mathbf{1}.
\end{eqnarray*}
 We have $\sigma(\bar{F}_{n})\subset \bar{F}_{n}$ for $n\in \N$. and
 $\sigma(w)\equiv w\ ({\rm mod} \bar{F}_{n-1})$ for $w\in \bar{F}_{n}$.
  Then it follows that $\sigma$ is bijective.
 Consequently, those vectors in (\ref{eq:3.5.}) are linearly  independent.
 \end{proof}

 Recall that $V_{\A(h)}$ is universal. As an immediate consequence of Proposition \ref{basis}, we have:

 \begin{cor}\label{basisV}
 For $n\in \N$, set
$$E_n={\rm span}\{u^{(1)}(m_{1})\cdots u^{(r)}(m_{r})\mathbf{1}\ |\  0\leq r\leq n, \  u^{(i)}\in U, \  m_{i}\in \Z\}
\subset V_{\A(h)}.$$
 Then $E_n$ for $n\geq 0$ form an ascending filtration of $V_{\A(h)}$, and for
$n\in \N$, $E_n$ has a basis consisting of the vectors
\begin{eqnarray}
e(-m_{1})\cdots e(-m_{r})f(-n_{1})\cdots f(-n_{s})\psi(-k_{1})\cdots \psi(-k_{l})\mathbf{1},
\end{eqnarray}
  where  $r,s,l\geq 0$ with $r+s+l\leq n$, and  $m_i,n_j,k_t\in \Z_{+}$ with
  $$m_{1}\geq\cdots\geq m_{r}\geq 1,\   n_{1}\geq\cdots\geq n_{s}\geq 1,\   k_{1}\geq\cdots\geq k_{l}\geq 1.$$
\end{cor}

Furthermore, we have:

 \begin{thm}\label{qva-main}
 Weak quantum vertex algebra $V_{\A(h)}$ is a non-degenerate quantum vertex algebra.
 \end{thm}

 \begin{proof} Recall that  for $n\in \N$,
 $$E_n={\rm span}\{u^{(1)}(m_{1})\cdots u^{(r)}(m_{r})\mathbf{1}\  |\   0\leq r\leq n,\ u^{(i)}\in U,\  m_{i}\in \Z\}. $$
 From \cite{Li-const} (Proposition 3.15),
 $\{E_{n}\}_{n\in\N}$ is an ascending filtration of $V_{\A(h)}$ such that $\mathbf{1}\in E_0$ and
 $u_{k}E_n\subseteq E_{m+n}$ for $u\in E_m, k\in \Z, m,n\in \N$. Consider the associated graded vector space
 $\mbox{Gr}_{E}(V_{\A(h)})=\coprod_{n\in\N}(E_{n}/E_{n-1})$ with $E_{-1}=0$.
 From \cite{Li-const} (Lemma 3.13),  $\mbox{Gr}_{E}(V_{\A(h)})$ is a $\Z$-graded nonlocal vertex algebra. Notice that $e=e(-1)\mathbf{1}\in E_{1},\ f=f(-1)\mathbf{1}\in E_{1},
 \ \psi=\psi(-1)\mathbf{1}\in E_{1}$. Set $\tilde{e}=e+E_{0}, \ \tilde{f}=f+E_{0},\  \tilde{\psi}=\psi+E_{0}\in E_{1}/E_{0}$. Since $e,f,\psi$ generate $V_{\A(h)}$ as a nonlocal vertex algebra, $\tilde{e}, \tilde{f}, \tilde{\psi}$ generate $\mbox{Gr}_{E}(V_{\A(h)})$ as a nonlocal vertex algebra, and we have
 \begin{eqnarray*}
 &&Y(\tilde{e},x_{1})Y(\tilde{e},x_{2})=h(x_2-x_1)Y(\tilde{e},x_{2})Y(\tilde{e},x_{1}),\\
 &&Y(\tilde{f},x_{1})Y(\tilde{f},x_{2})=h(x_1-x_2)Y(\tilde{f},x_{2})Y(\tilde{f},x_{1}),\\
&&Y(\tilde{\psi},x_{1})Y(\tilde{e},x_{2})=h(x_2-x_1)Y(\tilde{e},x_{2})Y(\tilde{\psi},x_{1}),\\
&&  Y(\tilde{\psi},x_{1})Y(\tilde{f},x_{2})=h(x_1-x_2)Y(\tilde{f},x_{2})Y(\tilde{\psi},x_{1}),
  \end{eqnarray*}
  $$Y(\tilde{e},x_{1})Y(\tilde{f},x_{2})=Y(\tilde{f},x_{2})Y(\tilde{e},x_{1}),\  \
  Y(\tilde{\psi},x_{1})Y(\tilde{\psi},x_{2})=Y(\tilde{\psi},x_{2})Y(\tilde{\psi},x_{1}).$$
  From Corollary \ref{basisV}, for each $n\in \N$, $E_{n}/E_{n-1}$ has a basis consisting of the vectors
 \begin{eqnarray*}
\tilde{e}(-m_{1})\cdots \tilde{e}(-m_{r})\tilde{f}(-n_{1})\cdots \tilde{f}(-n_{s})\tilde{\psi}(-k_{1})\cdots \tilde{\psi}(-k_{l})\mathbf{1},
\end{eqnarray*}
 where $r,s,l\geq 0$ with $r+s+l=n$, and
 $$m_{1}\geq\cdots\geq m_{r}\geq 1,\   n_{1}\geq\cdots\geq n_{s}\geq 1, \   k_{1}\geq\cdots\geq k_{l}\geq 1.$$
   From  \cite{KL} (Proposition 4.11), $\mbox{Gr}_{E}(V_{\A(h)})$ is nondegenerate. Then
  from  \cite{Li-const} (Proposition 3.14), $V_{\A(h)}$ is nondegenerate.
  Therefore,  $V_{\A(h)}$ is a nondegenerate  quantum vertex algebra.
 \end{proof}

\begin{rmk}
{\em Note that  the associated graded vertex algebra of  $V_{\A(h)}$ is isomorphic to $V_{L(\g)}$,
which is a graded vertex algebra,  but $V_{\A(h)}$ itself is not a graded nonlocal vertex algebra.}
\end{rmk}

Next, we study the case with  $h(z)=-q(-z)q(z)^{-1}$.
In this case,  we make use of a concrete vertex superalgebra instead of a vertex algebra.
Let $\mathfrak{g}^{(s)}=\mathfrak{g}^{(s)}_{0}\oplus \mathfrak{g}^{(s)}_{1}$ be the Lie superalgebra with
 even part $\mathfrak{g}^{(s)}_{0}=\C \tilde{\psi}$ and odd part
 $\mathfrak{g}^{(s)}_{1}=\C \tilde{e}\oplus  \C\tilde{f}$, and with relations
 \begin{eqnarray}
 [\tilde{e}, \tilde{e}]_{+}=0=[\tilde{f},  \tilde{f}]_{+},\   \    [\tilde{e},\tilde{f}]_{+}=\tilde{\psi},\    \
  [\tilde{\psi}, \tilde{e}]=0=[\tilde{\psi},  \tilde{f}].
  \end{eqnarray}
(It is straightforward to check that this indeed defines a Lie superalgebra.)
We then have a loop Lie superalgebra $L(\g^{(s)})=\mathfrak{g^{(s)}}\otimes \C[t,t^{-1}]$.

Form an induced $L(\g^{(s)})$-module
$V_{L(\g^{(s)})}=U(L(\g^{(s)}))\otimes_{U(\mathfrak{g^{(s)}\otimes\C[t]})}\C$,
where $\C$ is considered as a trivial module for $\mathfrak{g^{(s)}}\otimes\C[t]$.
   Then $V_{L(\g^{(s)})}$ has a basis
$$\bar{e}(-m_{1})\cdots\bar{e}(-m_{r})\bar{f}(-n_{1})\cdots\bar{f}(-n_{s})\bar{\psi}(-k_{1})\cdots\bar{\psi}(-k_{l})\mathbf{1}$$
for $r,s,l\geq 0$ with $m_{1}>\cdots>m_{r}\geq 1, n_{1}>\cdots >n_{s}\geq 1, k_{1}\geq\cdots\geq k_{l}\geq 1$.
Furthermore, $V_{L(\g^{(s)})}$ has a canonical vertex superalgebra structure.
By using this vertex superalgebra, we obtain our main results which are summarized as follows:

 \begin{prop}\label{basisV'}
Assume $h(z)=-q(-z)q(z)^{-1}$ with $q(z)\in \C[[z]]$ such that $q(0)=1$.  For  $n\in \N$, set
$$E_n={\rm span}\{u^{(1)}(m_{1})\cdots u^{(r)}(m_{r})\mathbf{1}\ |\  0\leq r\leq n, \  u^{(i)}\in \{e,f,\psi\}, \  m_{i}\in \Z\}.$$
Then $E_n$ for $n\geq 0$ form an ascending filtration of $V_{\A(h)}$, and for each nonnegative integer $n$, $E_n$ has a basis consisting of the vectors
\begin{eqnarray}
e(-m_{1})\cdots e(-m_{r})f(-n_{1})\cdots f(-n_{s})\psi(-k_{1})\cdots \psi(-k_{l})\mathbf{1},
\end{eqnarray}
where $r,s,l\geq 0$ with $r+s+l\leq n$, and
$$m_{1}>\cdots> m_{r}\geq 1,\  \   n_{1}>\cdots> n_{s}\geq 1, \   \   k_{1}\geq\cdots\geq k_{l}\geq 1.$$
Furthermore,  the weak quantum vertex algebra $V_{\A(h)}$ is a nondegenerate quantum vertex algebra.
\end{prop}

\section{Associative algebra $\tilde{\mathcal{A}}(g)$ and quantum vertex algebras $V_{\A(h)}$}

In this section, we study the quantum vertex algebra $V_{\A(h)}$
with $h(z)=g(e^{z})$, where $g(z)$ is a rational function such that $g(z)g(1/z)=1$.
More specifically, we introduce a new associative algebra $\tilde{\A}(g)$ and establish an isomorphism between
the category of restricted $\tilde{\A}(g)$-modules and
 the category of $\phi$-coordinated $V_{\A(h)}$-modules.

We begin by recalling some basics on formal calculus from \cite{Li-phi}.
Let $\C(z)$ denote the field of rational functions.
Denote by $\iota_{z,0}$ (resp. $\iota_{z,\infty}$) the field embedding of $\C(z)$ into $\C((z))$ (resp. $\C((z^{-1}))$),
which is the unique extension of
the embedding of $\C[z]$ into $\C((z))$ (resp. $\C((z^{-1}))$).
That is, for any rational function $k(z)$, $\iota_{z,0}(k(z))$ (resp. $\iota_{z,\infty}(k(z))$)
is the formal Laurent series expansion of $k(z)$ at $z=0$ (resp. $z=\infty$).

Denote by $F_{\C[[x]]}$ and $F_{\C[[w,z]]}$ the fields of fractions of rings $\C[[x]]$ and $\C[[w,z]]$, respectively.
We extend the domain of the field embedding $\iota_{x,0}$ from $\C(x)$ to $F_{\C[[x]]}$.
On the other hand, let $\iota_{w,z}$ denote the field embedding of $F_{\C[[w,z]]}$ into $\C((w))((z))$,
which is the unique extension of the canonical embedding of the ring $\C[[w,z]]$ into
$\C((w))((z))$.

Note that for any polynomial $p(z)$, $p(e^{x})\in \C[[x]]$ and that
 the assignment $p(z)\mapsto p(e^x)$ gives a ring embedding of $\C[z]$ into $\C[[x]]$.
 Furthermore, for $k(z)\in \C(z)$, we have $k(e^x)\in F_{\C[[x]]}$
and $\iota_{x,0}(k(e^x))\in \C((x))$.
On the other hand, we have
$$p(e^{x-y})\in \C[[x,y]]\   \   \   \mbox{ for }p(z)\in \C[z].$$
Furthermore, for any rational function $k(z)$,  we have
$$k(e^{x-y})\in F_{\C[[x,y]]}\   \mbox{ and  }\    \iota_{x,y}(k(e^{x-y}))\in \C((x))((y)).$$

In this section, we shall consider rational functions $g(z)$ such that $g(z)g(1/z)=1$.
In this respect, we have the following result which might be known somewhere:

\begin{lem}\label{g(z)-form}
Let $g(z)$ be a rational function. Then $g(z)g(1/z)=1$ if and only if
$$g(z)=\pm z^{l}p(z)/\tilde{p}(z),$$ where $l\in \Z,\ p(z)\in \C[z]$ with $p(0)\ne 0$ and
$\tilde{p}(z)=z^n p(1/z)\in \C[z]$ with $n=\deg p(z)$.
\end{lem}

\begin{proof} The ``if'' part is clearly true. Now, we assume $g(z)g(1/z)=1$.
Write $g(z)=z^{l}p(z)/q(z)$, where $l\in \Z$ and $p(z),q(z)$ are relatively prime polynomials with
$p(0), q(0)\ne 0$.
Set $m=\deg p(z)$ and $n=\deg q(z)$. As $g(z)g(1/z)=1$, we have
$p(z)p(1/z)=q(z)q(1/z)$, which gives
\begin{eqnarray}\label{left-right}
z^{n}p(z)(z^{m}p(1/z))=z^{m}q(z)(z^{n}q(1/z)).
\end{eqnarray}
Notice that $z^{m}p(1/z)$ and $z^{n}q(1/z)$  are polynomials of degrees $m$ and $n$, respectively.
It follows that $m=n$. Furthermore, since $p(z)$ and $q(z)$ are relatively prime,
from (\ref{left-right}) we get $q(z)| z^{n}p(1/z)$.
As $\deg q(z)=n=\deg z^np(1/z)$, we have
$q(z)=\alpha z^{n}p(1/z)$ for some nonzero complex number $\alpha$. Thus $g(z)=\alpha^{-1}p(z)/\tilde{p}(z)$.
Using $g(z)g(1/z)=1$ again, we get $\alpha^2=1$. Therefore, we have
$g(z)=\pm z^{l}p(z)/\tilde{p}(z),$ as desired.
\end{proof}

\begin{rmk}
{\em Taking $l=0$ and $p(z)=(z-q_1)(z-q_2)(z-q_3)$ with $q_i\ne 0,\pm 1$ for $i=1,2,3$
in Lemma \ref{g(z)-form}, we get
$$g(z)=\pm \frac{(z-q_1)(z-q_2)(z-q_3)}{(1-q_1z)(1-q_2z)(1-q_3 z)}.$$
This gives the rational function $h(z)$ that was used in \cite{bfmzz}.}
\end{rmk}

Using Lemma \ref{g(z)-form} we have:

 \begin{lem}\label{h(x)-factorization}
 Let $g(z)$ be any rational function such that $g(z)g(1/z)=1$ and set $h(x)=\iota_{x,0}(g(e^{x}))$. Then
 there exists $q(x)\in \C[[x]]$ such that $h(x)=\pm q(x)q(-x)^{-1}$.
 \end{lem}

 \begin{proof} From Lemma \ref{g(z)-form}, we have $g(z)=\pm z^{l}p(z)/\tilde{p}(z)$, where $l\in \Z$,
 $p(z)\in \C[z]$ such that $p(0)\ne 0$ and where $\tilde{p}(z)=z^{n}p(1/z)$ with $n=\deg p(z)$.
 Furthermore, $p(z)$ and $\tilde{p}(z)$ are relatively prime from the proof.
 Notice that we may choose $p(z)$ to be monic.
 Then $p(z)=(z-q_1)(z-q_2)\cdots (z-q_n)$ with $q_i\ne 0$  for $1\le i\le n$, and
 $$\tilde{p}(z)=z^{n}p(1/z)=(1-q_1z)(1-q_2z)\cdots (1-q_n z).$$
Thus
 \begin{eqnarray}
 g(z)=\pm z^{l}\frac{(z-q_1)(z-q_2)\cdots (z-q_n)}{(1-q_1z)(1-q_2z)\cdots (1-q_n z)}.
 \end{eqnarray}
As $p(z)$ and $\tilde{p}(z)$ are relatively prime, we have $q_i\ne q_j^{-1}$ for $1\le i,j\le n$.
  Setting
 $$q(x)=e^{\frac{1}{2}lx}\cdot (e^{x}-q_1)(e^x-q_2)\cdots (e^x-q_n )e^{-\frac{1}{2}nx}\in \C[[x]],$$
 we obtain
 \begin{eqnarray*}
 q(x)q(-x)^{-1}=\frac{e^{\frac{1}{2}lx}}{e^{-\frac{1}{2}lx}}\cdot
 \frac{(e^{x}-q_1)(e^x-q_2)\cdots (e^x-q_n )e^{-\frac{1}{2}nx}}{(e^{-x}-q_1)(e^{-x}-q_2)
 \cdots (e^{-x}-q_n)e^{\frac{1}{2}nx}}
 =\pm g(e^{x}),
 \end{eqnarray*}
 as desired.
\end{proof}

\begin{rmk}\label{nondegenerateQVA}
{\em Assume $h(x)=\iota_{x,0}(g(e^{x}))$, where $g(z)$ is a rational function such that
$g(z)g(1/z)=1$. Combining Lemma \ref{h(x)-factorization} with
Theorem \ref{qva-main} and Proposition \ref{basisV'}, we have that
$V_{\A(h)}$ is a nondegenerate quantum vertex algebra.}
\end{rmk}

From now on, we  fix a rational function  $g(z)$ such that $g(z)g(1/z)=1$ and set
 \begin{eqnarray}
 h(x)=\iota_{x,0}(g(e^{x}))\in \C((x)).
 \end{eqnarray}
Notice that with $g(z)g(1/z)=1$ we have $g(1)^{2}=1$. This especially implies that $g(z)$ is analytic at $z=1$.
Thus $h(x)\in \C[[x]]$ with $h(x)h(-x)=1$.

\begin{definition}\label{def}
{\em Define $\tilde{\A}(g)$ to be the associative unital algebra over $\mathbb{C}$, generated by
$$E_{n},\   F_{n},\    \Psi_{n}\   \   (n\in \mathbb{Z}),$$
 subject to relations
\begin{eqnarray}
&&E(z)E(w)=\iota_{w,z}(g(w/z))E(w)E(z), \    \   \
F(z)F(w)=\iota_{w,z}(g(z/w))F(w)F(z),\label{eq:5.1}\\
&&\Psi(z)E(w)=\iota_{w,z}(g(w/z))E(w)\Psi(z), \    \   \
\Psi(z)F(w)=\iota_{w,z}(g(z/w))F(w)\Psi(z),\label{eq:5.2}\\
&&[E(z),F(w)]=\delta\left(\frac{w}{z}\right)\Psi(w),\label{eq:5.3}\\
&&[\Psi(z),\Psi(w)]=0,\label{eq:5.4}
\end{eqnarray}
where
$$E(z)=\sum_{n\in \mathbb{Z}}E_{n}z^{-n},\   \   \   F(z)=\sum_{n\in \mathbb{Z}}F_{n}z^{-n},\   \  \
\Psi(z)=\sum_{n\in \mathbb{Z}}\Psi_{n}z^{-n}.$$}
\end{definition}

Write
$$\iota_{z,0}(g(z))=\sum_{l\ge p}g_{l}z^{l},\   \  \   \   \iota_{z,0}(g(1/z))=\sum_{l\ge q}\tilde{g}_{l}z^{l},$$
where $p$ and $q$ are integers. We have
$$\iota_{w,z}(g(z/w))=\sum_{l\ge p}g_{l}z^{l}w^{-l},\    \   \
\iota_{w,z}(g(w/z))=\sum_{l\ge q }\tilde{g}_{l}z^{l}w^{-l}.$$
The defining relations of $\tilde{\A}(g)$ can be written in terms of components as
\begin{eqnarray}\label{eA(h)relations}
&&E_{m}E_{n}=\sum_{l\ge q}\tilde{g}_{l}E_{n-l}E_{m+l},\   \   \   \  F_{m}F_{n}=\sum_{l\ge p}g_{l}F_{n-l}F_{m+l},
\nonumber\\
&&\Psi_{m}E_{n}=\sum_{l\ge q}\tilde{g}_{l}E_{n-l}\Psi_{m+l},\   \   \   \  \Psi_{m}F_{n}=\sum_{l\ge p}g_{l}F_{n-l}\Psi_{m+l},
\nonumber\\
&&[E_{m},F_{n}]=\Psi_{m+n},\   \    \   \   [\Psi_{m},\Psi_{n}]=0
\end{eqnarray}
for $m,n\in \Z$.  From this, we see that $\tilde{\A}(g)$  is a $\Z$-graded algebra with
\begin{eqnarray}
\deg E_{n}=\deg F_{n}=\deg \Psi_{n}=-n\   \   \  \mbox{ for }n\in \Z.
\end{eqnarray}
Then we define the notions of  (lower truncated) $\Z$-graded $\tilde{\A}(g)$-module and
$\N$-graded $\tilde{\A}(g)$-module in the obvious way.


\begin{definition}
{\em An $\tilde{\A}(g)$-module $W$ is said to be {\em restricted} if for every $w\in W$,
$E_{n}w=0,\ F_{n}w=0$ and $\Psi_{n}w=0$ for $n$ sufficiently large, or namely,
if $E(x),\ F(x)$ and $\Psi(x)\in \E(W)$.}
\end{definition}

Let $W$ be a restricted $\tilde{\A}(g)$-module.  Set $U_{W}=\{E(x),F(x),\Psi(x)\}.$
From relations (\ref{eq:5.1})--(\ref{eq:5.4}), it can be readily seen
that $U_{W}$ is an $S_{trig}$-local subset
 of ${\mathcal{E}}(W)$. In view of Theorem \ref{thm2.2}, $U_W$ generates
a weak quantum vertex algebra $\langle U_W\rangle_{e}$ in ${\mathcal{E}}(W).$

The following gives a connection between  algebras $\tilde{\A}(g)$ and $\A(h)$:

\begin{prop}\label{vacuum module}
Let $W$ be a restricted $\tilde{\A}(g)$-module and let $\langle U_W\rangle_{e}$ be the weak quantum vertex algebra generated by the subset $U_{W}=\{E(x),F(x),\Psi(x)\}$ of ${\mathcal{E}}(W).$
Then $\langle U_W\rangle_{e}$ is an $\A(h)$-module with $e(z),\ f(z),$ and $\psi(z)$ acting as
 $Y_{\mathcal{E}}^{e}(E(x),z)$, $Y_{\mathcal{E}}^{e}(F(x),z),$ and $Y_{\mathcal{E}}^{e}(\Psi(x),z)$, respectively.
 Moreover, $(\langle U_W\rangle_{e},\textbf{1}_W)$ is a vacuum  $\A(h)$-module.
\end{prop}

 \begin{proof} With the relations (\ref{eq:5.1})-(\ref{eq:5.4}), from Proposition 5.3 in \cite{Li-phi}  we have
 \begin{eqnarray}
&&Y_{\mathcal{E}}^{e}(E(x),z_{1})Y_{\mathcal{E}}^{e}(E(x),z_{2})=\iota_{z_2,z_1}(g(e^{z_2-z_1}))Y_{\mathcal{E}}^{e}(E(x),z_{2})Y_{\mathcal{E}}^{e}(E(x),z_{1}), \nonumber\\
&&Y_{\mathcal{E}}^{e}(F(x),z_{1})Y_{\mathcal{E}}^{e}(F(x),z_{2})=\iota_{z_2,z_1}(g(e^{z_1-z_2}))Y_{\mathcal{E}}^{e}(F(x),z_{2})Y_{\mathcal{E}}^{e}(F(x),z_{1}),\nonumber \\
&&Y_{\mathcal{E}}^{e}(\Psi(x),z_{1})Y_{\mathcal{E}}^{e}(E(x),z_{2})=\iota_{z_2,z_1}(g(e^{z_2-z_1}))Y_{\mathcal{E}}^{e}(E(x),z_{2})Y_{\mathcal{E}}^{e}(\Psi(x),z_{1}),\nonumber \\
&&Y_{\mathcal{E}}^{e}(\Psi(x),z_{1})Y_{\mathcal{E}}^{e}(F(x),z_{2})
=\iota_{z_2,z_1}(g(e^{z_1-z_2}))Y_{\mathcal{E}}^{e}(F(x),z_{2})Y_{\mathcal{E}}^{e}(\Psi(x),z_{1}), \nonumber\\
&&(z_2-z_1)Y_{\mathcal{E}}^{e}(E(x),z_{1})Y_{\mathcal{E}}^{e}(F(x),z_{2})=
(z_2-z_1)Y_{\mathcal{E}}^{e}(F(x),z_{2})Y_{\mathcal{E}}^{e}(E(x),z_{1}), \label{eq:5.5}\\
&&[Y_{\mathcal{E}}^{e}(\Psi(x),z_{1}),Y_{\mathcal{E}}^{e}(\Psi(x),z_{2})]=0.\nonumber
\end{eqnarray}
Furthermore, with the $\mathcal{S}$-locality relation (\ref{eq:5.5})  we have
\begin{eqnarray}
 && z_{0}^{-1}\delta\left(\frac{z_{1}-z_{2}}{z_{0}}\right)Y_{\mathcal{E}}^{e}(E(z),z_{1})Y_{\mathcal{E}}^{e}(F(z),z_{2})\nonumber\\
   &&\  \   \  \   -z_{0}^{-1}\delta\left(\frac{z_{2}-z_{1}}{-z_{0}}\right)
   Y_{\mathcal{E}}^{e}(F(z),z_{2})Y_{\mathcal{E}}^{e}(E(z),z_{1})
                                                            \nonumber\\
&=&z_{2}^{-1}\delta\left(\frac{z_{1}-z_{0}}{z_{2}}\right)
Y_{\mathcal{E}}^{e}(Y_{\mathcal{E}}^{e}(E(z),z_0)F(z),z_{2}).\label{eq:5.6}
\end{eqnarray}
With relations (\ref{eq:5.1})-(\ref{eq:5.4}), from Lemma 6.7 of \cite{Li-phi} we have
$E(z)^{e}_{n}F(z)=0$ for $n\geq 1$ and
\begin{eqnarray*}
E(z)^{e}_{0}F(z)
&=&\mbox{Res}_{z_{1}}\left(z^{-1}E(z_{1})F(z)-z^{-1}F(z)E(z_1)\right) \nonumber\\
 &=&\mbox{Res}_{z_{1}}z^{-1}\delta\left(\frac{z_1}{z}\right)\Psi(z)\nonumber\\
          &=&\Psi(z).
\end{eqnarray*}
Then using (\ref{eq:5.6}), we get
 \begin{eqnarray*}
 &&Y_{\mathcal{E}}^{e}(E(z),z_{1})Y_{\mathcal{E}}^{e}(F(z),z_{2})
-Y_{\mathcal{E}}^{e}(F(z),z_{2})Y_{\mathcal{E}}^{e}(E(z),z_{1})\nonumber\\
       &=& z_{2}^{-1}\delta\left(\frac{z_{1}}{z_{2}}\right)Y_{\mathcal{E}}^{e}(E(z)_{0}^{e}F(z),z_{2})\nonumber\\
       &=& z_{2}^{-1}\delta\left(\frac{z_{1}}{z_{2}}\right)Y_{\mathcal{E}}^{e}(\Psi(z),z_{2}).
\end{eqnarray*}
It follows that $\langle U_W\rangle_{e}$ is an $\A(h)$-module with
$e(z)$, $f(z),$ and $\psi(z)$ acting as
$Y_{\mathcal{E}}^{e}(E(x),z)$, $Y_{\mathcal{E}}^{e}(F(x),z),$ and $Y_{\mathcal{E}}^{e}(\Psi(x),z)$, respectively.
Since  $\langle U_W\rangle_{e}$ is generated by
$\{E(x),F(x),\Psi(x)\}$ as a nonlocal vertex algebra,
                   we see that $\langle U_W\rangle_{e}$ is generated from $\textbf{1}_W$ as an $\A(h)$-module.
                  As $\textbf{1}_{W}$ is the vacuum vector of
                   the weak quantum vertex algebra $\langle U_W\rangle_{e}$,
                   we have $E(x)_{n}^{e}\textbf{1}_W=0$, $F(x)_{n}^{e}\textbf{1}_W=0,$ and $\Psi(x)_{n}^{e}\textbf{1}_W=0$
                   for $n\geq 0$.
Therefore, $(\langle U_W\rangle_{e},\textbf{1}_W)$ is a vacuum $\A(h)$-module.
\end{proof}

\begin{rmk}\label{rmk:5.1}
{\em Recall that $V_{\A(h)}$ is a weak quantum vertex algebra with  a set of generators $\{e,f,\psi\}$ and with
$Y(e,x)=e(x)$, $Y(f,x)=f(x)$ and
$Y(\psi,x)=\psi(x)$, and
\begin{eqnarray*}
 && Y(e,x_1)Y(e,x_2)=\iota_{x_2,x_1}(g(e^{x_2-x_1}))Y(e,x_2)Y(e,x_1),\\
 && Y(f,x_1)Y(f,x_2)=\iota_{x_2,x_1}(g(e^{x_1-x_2}))Y(f,x_2)Y(f,x_1),\\
  && Y(\psi,x_1)Y(e,x_2)=\iota_{x_2,x_1}(g(e^{x_2-x_1}))Y(e,x_2)Y(\psi,x_1),\\
    && Y(\psi,z_1)Y(f,z_2)=\iota_{x_2,x_1}(g(e^{x_1-x_2}))Y(f,x_2)Y(\psi,x_1),\\
  &&[Y(e,x_1),Y(f,x_2)]=x_1^{-1}\delta\left(\frac{x_2}{x_1}\right)Y(\psi,x_2),\\
&&[Y(\psi,x_1),Y(\psi,x_2)]=0,
\end{eqnarray*}
where $h(x)=\iota_{x,0}(g(e^x))\in \C((x))$.
From these relations, we have (see \cite{Li-nonlocal})
\begin{eqnarray*}
&&e_{n}e=0,\ f_{n}f=0,\ \psi_{n}e=0, \ \psi_{n}f=0,\  \psi_{n}\psi=0\   \   \   \mbox{ for }n\ge 0,\\
&&e_{0}f=\psi \  \mbox{ and }\  e_{n}f=0 \   \   \   \mbox{ for }n\ge 1.
\end{eqnarray*}}
\end{rmk}

Now, we are in a position to present the main result of this section:

\begin{thm}\label{thm5.3}
Let $W$ be a restricted $\tilde{\A}(g)$-module.
Then there exists a $\phi$-coordinated $V_{\A(h)}$-module structure
$Y_{W}(\cdot,x)$ on $W$, which is uniquely determined by $Y_{W}(e,x)=E(x),$ $Y_{W}(f,x)=F(x)$ and
$Y_{W}(\psi,x)=\Psi(x)$.
On the other hand, suppose $(W,Y_W)$ is a $\phi$-coordinated $V_{\A(h)}$-module. Then
$W$ is a restricted $\tilde{\A}(g)$-module with $E(x)=Y_{W}(e,x),$ $F(x)=Y_{W}(f,x)$ and $\Psi(x)=Y_{W}(\psi,x)$.
\end{thm}

 \begin{proof} Let $W$ be a restricted $\tilde{\A}(g)$-module. From Theorem \ref{vacuum module},
 the weak quantum vertex algebra $\langle U_W\rangle_{e}$ generated by $U_W=\{E(x),F(x),\Psi(x)\}$
  is a vacuum $\A(h)$-module with $e_{n},$ $f_{n},$ and $\psi_{n}$ for $n\in \Z$ acting as $E(x)_{n}^{e}$,  $F(x)_{n}^{e}$, and
  $\Psi(x)_{n}^{e}$, respectively.
Since the vacuum $\A(h)$-module ($V_{\A(h)}$, ${\bf 1}$)
  is universal, there exists an $\A(h)$-module homomorphism
 $\sigma$ from $V_{\A(h)}$ to $\langle U_W\rangle_{e}$, sending $\textbf{1}$ to $1_{W}$.
 We have
 $$\sigma(e)=\sigma(e_{-1}{\bf{1}})=E(x)_{-1}^{e}1_{W}=E(x),$$
 $$\sigma(f)=\sigma(f_{-1}{\bf{1}})=F(x)_{-1}^{e}1_{W}=F(x),$$
 $$\sigma(\psi)=\sigma(\psi_{-1}{\bf{1}})=\Psi(x)_{-1}^{e}1_{W}=\Psi(x),$$
 and
 \begin{eqnarray*}
 &&\sigma(Y(e,z)v)=\sigma(e(z)v)=Y_{\mathcal{E}}^{e}(E(x),z)\sigma(v)=Y_{\mathcal{E}}^{e}(\sigma(e),z)\sigma(v)\\
 &&\sigma(Y(f,z)v)=\sigma(f(z)v)=Y_{\mathcal{E}}^{e}(F(x),z)\sigma(v)=Y_{\mathcal{E}}^{e}(\sigma(f),z)\sigma(v)\\
 &&\sigma(Y(\psi,z)v)=\sigma(\psi(z)v)=Y_{\mathcal{E}}^{e}(\Psi(x),z)\sigma(v)
 =Y_{\mathcal{E}}^{e}(\sigma(\psi),z)\sigma(v)
 \end{eqnarray*}
 for $v\in V_{\A(h)}$.
Since $V_{\A(h)}$ is generated by $e,f,\psi$ as a nonlocal vertex algebra,
it follows that $\sigma$ is a homomorphism of nonlocal vertex algebras.
 As  $W$ is a canonical $\phi$-coordinated module for the
weak quantum vertex algebra  $\langle U_W\rangle_{e}$, it follows that
  $W$ is a $\phi$-coordinated $V_{\A(h)}$-module with
  $Y_{W}(e,z)=E(z)$, $Y_{W}(f,z)=F(z)$ and $Y_{W}(\psi,z)=\Psi(z)$.

On the other hand, let $W$ be a $\phi$-coordinated $V_{\A(h)}$-module.
  From Propositions 5.6 and  5.9 of \cite{Li-phi}, we have
\begin{eqnarray*}
 && Y_{W}(e,z_1)Y_{W}(e,z_2)-\iota_{z_2,z_1}(g(z_{2}/z_{1}))Y_{W}(e,z_2)Y_{W}(e,z_1)=0,\\
 && Y_{W}(f,z_1)Y_{W}(f,z_2)-\iota_{z_2,z_1}(g(z_{1}/z_{2}))Y_{W}(f,z_2)Y_{W}(f,z_1)=0,\\
 && Y_{W}(\psi,z_1)Y_{W}(e,z_2)-\iota_{z_2,z_1}(g(z_{2}/z_{1}))Y_{W}(e,z_2)Y_{W}(\psi,z_1)=0,\\
  &&Y_{W}(\psi,z_1)Y_{W}(f,z_2)-\iota_{z_2,z_1}(g(z_{1}/z_{2}))Y_{W}(f,z_2)Y_{W}(\psi,z_1)=0,\\
 && Y_{W}(e,z_1)Y_{W}(f,z_2)-Y_{W}(f,z_2)Y_{W}(e,z_1)\\
 &=& \mbox{Res}_{z_{0}}z_{1}^{-1}\delta\left(\frac{z_2e^{z_{0}}}{z_{1}}\right)z_{2}e^{z_{0}}Y_{W}(Y(e,z_0)f,z_2)\\
 &=& \delta\left(\frac{z_2}{z_1}\right)Y_{W}(\psi,z_{2}),
\end{eqnarray*}
where we use the fact that $e_{n}f=0$ for $n\ge 1$ and $e_{0}f=\psi$
(see Remark \ref{rmk:5.1}).
This shows that $W$ is an $\tilde{\A}(g)$-module with $E(z)=Y_{W}(e,z),$ $F(z)=Y_{W}(f,z)$ and $\Psi(z)=Y_{W}(\psi,z)$.
On the other hand, with $W$ a $\phi$-coordinated $V_{\A(h)}$-module we have
$Y_{W}(e,z)$, $Y_{W}(f,z)$, $Y_{W}(f,z)\in {\mathcal{E}}(W)$.
Therefore, $W$ is a restricted $\tilde{\A}(g)$-module.
\end{proof}

In view of Theorem \ref{thm5.3}, classifying irreducible $\N$-graded $\phi$-coordinated $V_{{\mathcal{A}}(h)}$-modules
amounts to classifying irreducible $\N$-graded $\tilde{\mathcal{A}}(g)$-modules.
In the following, we shall classify irreducible $\N$-graded $\tilde{\mathcal{A}}(g)$-modules for $g$ of a certain form.

For the rest of this section, we assume  such that  $g(z)$ is analytic at $z=0$ with $g(0)\ne 0$.
(This amounts to that $l=0$ in Lemma \ref{g(z)-form}.)
This implies that $g(z)$ is also analytic at $z=\infty$ with $g(0)g(\infty)=1$. Then we have
$$\iota_{z,0}(g(z)),\    \    \   \   \iota_{z,0}(g(1/z)) \in \C[[z]],$$
so that
\begin{eqnarray}
\iota_{z,0}(g(z))=\sum_{l\ge 0}g_{l}z^{l},\   \  \   \   \iota_{z,0}(g(1/z))=\sum_{l\ge 0}\tilde{g}_{l}z^{l}.
\end{eqnarray}
 Set $\alpha=g(0)=g_0$. Then $g(\infty)=\tilde{g}_{0}=\alpha^{-1}$.

\begin{definition}
{\em  Define  $A[\alpha]$ to be the associative unital algebra over $\C$ with generators $E[0], F[0], \Psi[0]$,
subject to relations
 \begin{eqnarray}
&&E[0]^{2}=\alpha^{-1}E[0]^{2}, \   \   \  \   F[0]^{2}=\alpha F[0]^{2}, \\
&&\Psi[0]E[0]=\alpha^{-1}E[0]\Psi[0],\   \    \    \Psi[0]F[0]=\alpha F[0]\Psi[0],\   \   \
  [E[0],F[0]]=\Psi[0].
\end{eqnarray}}
\end{definition}

It can be readily seen that $A[\alpha]$ becomes an $\N$-graded algebra by defining
 $\deg E[0]=1=\deg F[0]$ and $\deg \Psi[0]=2$.
We see that $A[\alpha]$ has a maximal ideal of codimension $1$, so that $A[\alpha]$ has one
$1$-dimensional  irreducible module which is unique up to isomorphism.

\begin{thm}\label{classification}
Let $W=\oplus_{n\in \N}W[n]$ be an $\N$-graded $\tilde{\mathcal{A}}(g)$-module
with $W[0]\ne 0$. Then $W[0]$ is an $A[\alpha]$-module with $E[0], F[0], \Psi[0]$ acting as
$E_{0},F_0,\Psi_0$, respectively. Furthermore,
if $W$ is irreducible, $W[0]$ is an irreducible $A[\alpha]$-module. On the other hand,
assume $W_1$ and $W_2$ are irreducible $\N$-graded  $\tilde{\mathcal{A}}(g)$-modules
with $W_1[0],\ W_2[0]\ne 0$. Then
$W_1$ and $W_2$ are  isomorphic
if and only if $W_1[0]$ and $W_2[0]$ are isomorphic $A[\alpha]$-modules.
\end{thm}

\begin{proof}  Suppose that $W=\oplus_{n\in \N}W[n]$ is an $\N$-graded $\tilde{\mathcal{A}}(g)$-module
with $W[0]\ne 0$. Note that $E_0,F_0,\Psi_{0}$ preserve $W[0]$ and that
$E_{n}W[0]=0, F_{n}W[0]=0, \Psi_{n}W[0]=0$ for $n>0$.
From commutation relations (\ref{eA(h)relations}), we see that the following relations hold on $W[0]$:
\begin{eqnarray*}
E_{0}^{2}=\alpha^{-1}E_0^{2}, \  \  \   F_{0}^{2}=\alpha F_0^{2}, \   \  \
\Psi_{0}E_0=\alpha^{-1}E_0\Psi_{0},\   \    \    \Psi_{0}F_0=\alpha F_0\Psi_{0},\   \   \
  [E_0,F_0]=\Psi_0.
\end{eqnarray*}
It follows that $W[0]$ is an $A[\alpha]$-module with $E[0], F[0], \Psi[0]$ acting as
$E_{0},F_0,\Psi_0$, respectively. If $U$ is an $A[\alpha]$-submodule of $W[0]$,
we see that $U$ generates a graded $\tilde{\A}(g)$-submodule of $W$ with $U$ as the degree zero subspace.
Thus $W[0]$ is an irreducible $A[\alpha]$-module if $W$ is irreducible.

Let $U$ be an $A[\alpha]$-module. We now introduce
an $\N$-graded module $M(U)$ of Verma type for $\tilde{\A}(g)$ by using a  tautological construction.
We start with the free (tensor) algebra $T$ generated by  $\check{E}_{n}$,
$\check{F}_{n}$, $\check{\Psi}_{n}$ for $n\in \Z$. Define
$$\deg \check{E}_{n}=-n,\  \  \deg \check{F}_{n}=-n,\  \  \deg \check{\Psi}_{n}=-n \   \   \mbox{ for }n\in \Z,$$
to make $T$ a $\Z$-graded algebra. For $n\in \Z$, let $T_{n}$ denote the homogeneous subspace of degree $n$.
We denote by $T^{0}$ the subalgebra generated by
$\check{E}_{0}$, $\check{F}_{0}$, $\check{\Psi}_{0}$. Notice that $T^{0}$ is a subalgebra of $T_{0}$ and it is also free.
Then there is a homomorphism from  $T^{0}$ onto $A[\alpha]$.
Consequently, $U$ is naturally a $T^{0}$-module. Set $T_{-}=\sum_{n\ge 1}T_{-n}$.
We see that $T_{-}+T^{0}$ is a subalgebra with $T_{-}$ as an ideal.
Let $T_{-}$ act on $U$ trivially to make $U$ a $(T_{-}+T^{0})$-module.
Next, form the induced module
$$M_{T}(U)=T\otimes_{(T_{-}+T^0)}U,$$
 which is an $\N$-graded $T$-module with $\deg U=0$.
It follows that for every $w\in M_{T}(U)$,
$$ \check{E}_{n}w=0,\  \   \check{F}_{n}w=0,\   \  \check{\Psi}_{n}w=0$$
for $n\in \Z$ sufficiently large. Noticing that there is a canonical homomorphism from $T$ onto $\tilde{\A}(g)$,
we then define $M(U)$ to be the quotient module of $M_{T}(U)$
modulo the submodule generated by the elements corresponding to the defining relations of $\tilde{\A}(g)$.
We see that $M(U)$ naturally becomes an $\N$-graded $\tilde{\A}(g)$-module.
Furthermore, $M(U)[0]$ is an $A[\alpha]$-module and the map sending $u$ to $1\otimes u$ is an $A[\alpha]$-module
homomorphism.

Assume that  there exists  an $\N$-graded $\tilde{\A}(g)$-module $W$ with $W[0]\simeq U$ as an $A[\alpha]$-module.
It follows that there is an $\tilde{\A}(g)$-module homomorphism from $M(U)$ to $W$. Consequently,
$M(U)[0]\simeq U$ as an $A[\alpha]$-module. If $U$ is an irreducible  $A[\alpha]$-module,
we see that $M(U)$ has a unique maximal graded submodule, so that $M(U)$ has a unique irreducible quotient
denoted by $L(U)$. Then the last assertion follows immediately.
\end{proof}

\begin{rmk}
{\em Let $U$ be an $A[\alpha]$-module. In the proof of Theorem \ref{classification},
we constructed  an $\N$-graded $\tilde{\A}(g)$-module $M(U)$ and we have a homomorphism
of $A[\alpha]$-modules from $U$ to $M(U)[0]$, sending $u$ to $1\otimes u$ for $u\in U$.
However,  we are unable to prove that this  is an isomorphism.  }
\end{rmk}

Consider the case where $\alpha\ne \pm1$. From the defining relations of $A[\alpha]$, we get
$$E_0^{2}=0,\  \  F_0^{2}=0,\  \  \Psi_0E_0=E_0\Psi_0=0,\  \  \Psi_0F_0=F_0\Psi_0=0,\  \  \Psi_0^{2}=0.$$
Furthermore, we have
$$(E_0F_0)^2=0,\  \  (F_0E_0)^2=0,\   \  (E_0F_0)(F_0E_0)=0=(F_0E_0)(E_0F_0).$$
It then follows that the linear span of $E_0,F_0,\Psi_0, E_0F_0$ is a nilpotent ideal of codimension $1$.
Consequently, up to isomorphism $A[\alpha]$ has exactly one  irreducible
module (with $E_0,F_0,\Psi_0$ acting trivially).

\begin{lem}
Assume $\alpha=-1$. Let $U$ be a $2$-dimensional vector space with a basis $\{ v_1,v_2\}$.
Then for any nonzero complex number $\lambda$, $U$ has an irreducible $A[-1]$-module structure with
 \begin{eqnarray}
 \Psi_0 v_1=\lambda v_1, \  \  \Psi_0 v_2=-\lambda v_2, \   \  F_0v_1=v_2,\   \   F_0v_2=0,\   \
 E_0v_1=0,\  \  E_0v_2=\lambda v_1.
 \end{eqnarray}
Denote this $A[-1]$-module by $U(\lambda)$.
Furthermore,  every finite-dimensional irreducible $A[-1]$-module is either $1$-dimensional or isomorphic to
such a $2$-dimensional module.
\end{lem}

\begin{proof} In this case, the defining relations of $A[\alpha]$ can be written as
\begin{eqnarray}
E_0^{2}=0=F_0^{2}, \  \  \   \Psi_{0}E_0=-E_0\Psi_{0},\   \  \
\Psi_{0}F_0=-F_0\Psi_{0},\   \   \
  [E_0,F_0]=\Psi_0.
\end{eqnarray}
From this, we see that $A[-1]$ has a basis $E_0^iF_0^j\Psi_0^{k}$ for $i,j=0,1,\ k=0,1,\dots$.
To show the first assertion, we need to show that
there is an algebra homomorphism $\rho_{\lambda}$ from $A[-1]$ (on)to the matrix algebra $M(2,\C)$
such that
\begin{eqnarray}
\rho_{\lambda}(E_0)=\lambda E_{12},\    \   \   \rho_{\lambda}(F_0)=E_{21},\   \   \   \rho_{\lambda}(\Psi_0)=
\lambda(E_{11}-E_{22}).
\end{eqnarray}
The proof of the latter is straightforward.

Suppose that $W$ is a finite-dimensional irreducible $A[-1]$-module.
It follows that $\Psi_0$ is semisimple on $W$.
(This is because there is an eigenvector
$v$ of $\Psi_0$ of some eigenvalue $\lambda$ and $W=A[-1]v$.) Suppose that $0$ is an eigenvalue of $\Psi_0$, i.e.,
there exists a nonzero vector $w$ such that $\Psi_0 w=0$. As $W=A[-1]w$, it follows that $\Psi_0=0$ on $W$.
We show $E[0]W=0=F[0]W$. Suppose $E_0w\ne 0$ for some vector $w\in W$.  Then $W=A[-1]E_0w$.
Note that $A[-1]E_0w$ is linearly spanned by $E_0w$ and $F_0E_0w$. As $E_0(E_0w)=0$ and $E_0(F_0E_0w)=0$,
we have $E_0A[-1]E_0w=0$. But $w\in A[-1]E_0w$ and $E_0w\ne 0$, a contradiction. Thus $E_0W=0$.
Similarly, we have $F_0W=0$. Then $W$ must be $1$-dimensional.

Now, assume that $0$ is not an eigenvalue of $\Psi_0$. As $[E_0,F_0]=\Psi_0$ and $\Psi_0W\ne 0$,
we have $E_0W\ne 0$.
Then there exist nonzero $w\in W$ and $\lambda\in \C^{\times}$ such that $E_0w=0$ and $\Psi_0w=\lambda w$.
With $W$ is irreducible, we have $W=A[-1]w=\C w+\C F_0w$, where
$E_0(F_0w)=(-1)\Psi_0w=-\lambda w$.
It follows that $W\simeq U(\lambda)$.
\end{proof}

\begin{rmk}
{\em Consider the case with $\alpha=1$. We see that $A[1]$  is isomorphic to the universal enveloping algebra of
the $3$-dimensional Heisenberg Lie algebra with base vectors $E_0,F_0,\Psi_0$ such that
$$[E_0,F_0]=\Psi_0,\   \   \  [\Psi_0,E_0]=0,\   \   \   [\Psi_0,F_0]=0.$$
There are at least three types of infinite-dimensional irreducible modules, which are
highest weight modules, lowest weight modules with $\Psi_0$ acting as nonzero scalars, and intermediate series modules
with $\Psi_0$ acting trivially.
A complete classification of irreducible $A[1]$-modules is not known.}
\end{rmk}

\end{document}